\documentclass[nohyperref]{article}

\usepackage{microtype}
\usepackage{graphicx}
\usepackage{subfigure}
\usepackage{booktabs} 

\usepackage{hyperref}

\usepackage[noend]{algorithmic}

\usepackage[accepted]{icml2022}

\usepackage{amsmath}
\usepackage{amssymb}
\usepackage{mathtools}
\usepackage{amsthm}
\usepackage{mathrsfs}

\usepackage[capitalize,noabbrev]{cleveref}

\theoremstyle{plain}
\newtheorem{theorem}{Theorem}[section]
\newtheorem{proposition}[theorem]{Proposition}
\newtheorem{lemma}[theorem]{Lemma}

\theoremstyle{definition}
\newtheorem{definition}[theorem]{Definition}

\theoremstyle{remark}

\DeclareMathOperator\sgn{sgn}
\newcommand{\utime}{\mathscr{T}}
\newcommand\scalemath[2]{\scalebox{#1}{\mbox{\ensuremath{\displaystyle #2}}}}

\usepackage[textsize=tiny]{todonotes}

\icmltitlerunning{Escaping Saddle Points Efficiently with Occupation-Time-Adapted Perturbations}

\begin{document}

\icmltitle{Escaping Saddle Points Efficiently with Occupation-Time-Adapted Perturbations} 



\icmlsetsymbol{equal}{*}

\begin{icmlauthorlist}
\icmlauthor{Xin Guo}{equal,uni1}
\icmlauthor{Jiequn Han}{equal,uni2}
\icmlauthor{Mahan Tajrobehkar}{equal,uni1}
\icmlauthor{Wenpin Tang}{equal,uni3}
\end{icmlauthorlist}

\icmlaffiliation{uni1}{Department of Industrial Engineering and Operations Research, UC Berkeley, Berkeley, CA, USA}
\icmlaffiliation{uni2}{Center of Computational Mathematics, Flatiron Institute, New York, NY, USA}
\icmlaffiliation{uni3}{Department of Industrial Engineering and Operations Research, Columbia University, New York City, NY, USA}

\icmlcorrespondingauthor{Xin Guo}{xinguo1@berkeley.edu}
\icmlcorrespondingauthor{Jiequn Han}{jiequnh@princeton.edu}
\icmlcorrespondingauthor{Mahan Tajrobehkar}{mahan\_tajrobehkar@berkeley.edu}
\icmlcorrespondingauthor{Wenpin Tang}{wt2319@columbia.edu}

\icmlkeywords{Machine Learning, ICML}

\vskip 0.3in



\printAffiliationsAndNotice{\icmlEqualContribution} 

\begin{abstract}
Motivated by the super-diffusivity of self-repelling random walk, which has roots in statistical physics, this paper develops a new perturbation mechanism for optimization algorithms. In this mechanism, perturbations are adapted to the history of states via the notion of occupation time. After integrating this mechanism into the framework of perturbed gradient descent (PGD) and perturbed accelerated gradient descent (PAGD), two new algorithms are proposed: perturbed gradient descent adapted to occupation time (PGDOT) and its accelerated version (PAGDOT). PGDOT and PAGDOT are shown to converge to second-order stationary points at least as fast as PGD and PAGD, respectively, and thus they are guaranteed to avoid getting stuck at non-degenerate saddle points. The theoretical analysis is corroborated by empirical studies in which the new algorithms consistently escape saddle points and outperform not only their counterparts, PGD and PAGD, but also other popular alternatives including stochastic gradient descent, Adam, AMSGrad, and RMSProp.
\end{abstract}

\section{Introduction}
\label{s1}

Gradient descent (GD), which dates back to \cite{Cauchy}, aims to minimize a function $f: \mathbb{R}^d \rightarrow \mathbb{R}$ via the iteration:
$
\pmb{x}_{t+1} = \pmb{x}_t - \eta \nabla f(\pmb{x}_t), t = 0, 1, 2, \ldots,
$
where $\eta >0$ is the step size  and $\nabla f$ is the gradient of $f$.
Due to its simple form and fine computational properties, GD and its variants (e.g., stochastic gradient descent) are essential for many machine learning tools: principle component analysis \cite{CLMW}, phase retrieval \cite{CLS}, and deep neural network \cite{RHW}, just to name a few.
In the era of data deluge, many problems are concerned with large-scale optimization in which the intrinsic dimension $d$ is large. 
GD turns out to be efficient in dealing with high-dimensional convex optimization, where the first-order stationary point 
$\nabla f(\pmb{x}) = 0 $ is necessarily the global minimum point.
Algorithmically, it involves finding a point with small gradient $|| \nabla f(\pmb{x})|| < \epsilon$. 
A classical result of \cite{Nesterov04} showed that the time required by GD to find such a point in a possibly non-convex problem is of order $\epsilon^{-2}$, independent of the dimension $d$.

In non-convex settings, applying GD will still lead to an approximate first-order stationary point.
However,  this is not sufficient: for non-convex functions, first-order stationary points can be either global minimum, local minimum, local maximum,  or saddle points.
As we will explain, saddle points are the main bottleneck for GD in many non-convex problems.
The goal of  this paper is 
therefore to develop efficient algorithms to escape saddle points in high-dimensional non-convex problems, and hence overcome the curse of dimensionality.

{\bf Escape local minima}: 
Inspired by annealing in metallurgy, \cite{KGV} developed simulated annealing to approximate the global minimum of a given function.
\cite{GH86} proposed a diffusion simulated annealing and proved that it converges to the set of global minimum points.
However, subsequent works \cite{HKS89, MSTW, Miclo, Mon18, TZ21} revealed that it might take an exponentially long time (of order $\exp(d)$) for diffusion simulated annealing to get close to the global minimum.
Some work, e.g., methods based on L\'evy flights \cite{Pav07} or Cuckoo's search \cite{YD09} showed empirically faster convergence to the global minimum. 
Yet  the theory of these approaches is far-fetched.
There are recent efforts in approximating the global minimum in non-convex problems via Langevin dynamics-based stochastic gradient descent \cite{RRT17, CDT20}, 
along with its variants using non-reversibility \cite{HW20}
and replica exchange \cite{CC19, DT21}.
Typically, these algorithms take polynomial time in the dimension $d$,
and thus may scale poorly when $d$ is large. 

{\bf Escape saddle points}: 
Fortunately, in many non-convex problems, it suffices to find a local minimum.  
Indeed, there has been a line of recent work arguing that local minima are less problematic, and that for many non-convex problems there are no spurious local minima.
That is, all local minima are comparable in value with the global minimum.
Examples include tensor decomposition \cite{GHJY, GLM18, GM17, SBRL19}, semidefinite programming \cite{BBV16, MMMO17}, dictionary learning \cite{SQW1}, phase retrieval \cite{SQW2}, robust regression \cite{MBM18}, low-rank matrix factorization \cite{BNS16, GJZ17, GLM16, DACS}, 
and certain classes of deep neural networks \cite{CHMBL, DVSH18, K16, KLD19, LSLS18, NH17, VBB19, WLL18}.
Nevertheless, as shown in \cite{DB14, DJL17, JJKN17}, saddle points may correspond to suboptimal solutions, and it may take exponentially long time to move from saddle points to a local minimum point. Meanwhile, it has  been observed in empirical studies \cite{dauphin2014identifying,swirszcz2016local} that GD and its variants with momentum such as Adam \cite{KB14} may be trapped in saddle points.

\cite{GHJY} took the first step to show that by adding noise at each iteration, GD can escape all saddle points in polynomial time.
Additionally, \cite{DLT18, LSJR16} proved that with random initialization, GD converges to a local minimizer.
Moreover, \cite{JG17} proposed the perturbed gradient descent (PGD) algorithm, which
 \cite{JNM18} further improved to the perturbed accelerated gradient descent (PAGD) algorithm.
They showed that PGD and PAGD are efficient -- the time complexity is almost independent of the dimension $d$.
See also \cite{JG19} for a summary of results in this direction.

\paragraph{Our idea.} 
Motivated by the ``fast exploration'' of self-repelling random walk, this paper develops a new perturbation mechanism by adapting the perturbations to the history of states.
Recall that \cite{JG17, JG19} used the following perturbation update when perturbation conditions hold:
\begin{align*}
\pmb{x}'_t  = \pmb{x}_t + \mbox{Unif}(B^d(\pmb{0}, r)), \quad \pmb{x}_{t+1} = \pmb{x}'_t - \eta \nabla f(\pmb{x}'_t),
\end{align*}
where $\mbox{Unif}(B^d(\pmb{0}, r))$ is a point picked uniformly in the ball of radius $r$. 
On the empirical side, \cite{NV15, ZL19} applied this idea of GD with noise to train deep neural networks.

Our idea is to replace $\mbox{Unif}(B^d(\pmb{0}, r))$ with non-uniform perturbations,
whose mechanism depends on the current state $\pmb{x}_t$ and  the history of states $\{\pmb{x}_s; \, s \le t\}$.
There are conceivably many ways to add non-uniform perturbation based on the current and previous states;
here we choose to adapt perturbations to the ``occupation time''.

The intuition is illustrated by the one-dimensional function $f(x) = x^3$ (see Figure \ref{fig:xcubed}).  

\begin{figure}[ht]
\centering
\includegraphics[width=0.3\columnwidth]{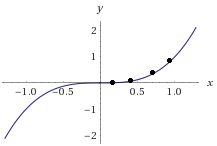}
\vspace{-0.05in}
\caption{Illustration of occupation-time-adapted perturbation using $f(x)=x^3$.}
\vspace{-0.1in}
\label{fig:xcubed}
\end{figure}

There is a saddle point at $0$, and imagine GD approaches $0$ from the right.
It can be shown that GD converges monotonically to a stationary point (see Appendix A).
The uniform perturbation will add noise with probability $1/2$ both to the right and to the left.
To the right, GD will again get stuck at the saddle point $0$. However, to the left, there is a possibility of escaping from $0$ and finding a local minimum ($-\infty$ in this case).
Therefore, it is reasonable to add noise with a larger probability to the left, since it has spent a long time on the right and has yet to explore the left side.

The previous intuition can be quantified via the notion of occupation times $L_t$ (the number of $\{x_s\}_{s<t}$ to the left of $x_t$) and $R_t$ (the number of $\{x_s\}_{s<t}$ to the right of $x_t$).
By definition, $R_t + L_t = t$, for each $t = 0,1, \ldots$.
If $L_t$ is larger, the perturbation will push the iterate $x_t$ to the right; and if $R_t$ is larger,  push to the left.
More precisely, 
\begin{equation}
\label{eq:1Dreverse}
x_{t+1} = \left\{\begin{array}{rl}
x_t - r \, \mbox{Unif}(0,1) & \mbox{with probability $p$}, \\ 
x_t + r \, \mbox{Unif}(0,1) & \mbox{with probability $1-p$},
\end{array}\right.
\end{equation}
where $p = \frac{w(R_t)}{w(L_t) + w(R_t)}$ and $w: \{0,1,\ldots\} \to (0,\infty)$ is an increasing weight function on the nonnegative integers (e.g., $w(n) = 1 + n^{\alpha}$ for $\alpha > 0$).

The dynamics \eqref{eq:1Dreverse} is closely related to the vertex-repelling random walk defined by
\begin{equation}
\label{eq:vrrw}
Z_{t+1} = \left\{ \begin{array}{rcl}
Z_t - 1 & \mbox{with probability } \frac{w(\widetilde{R}_t)}{w(\widetilde{L}_t) + w(\widetilde{R}_t)}, \\ 
Z_t + 1 & \mbox{with probability } \frac{w(\widetilde{L}_t)}{w(\widetilde{L}_t) + w(\widetilde{R}_t)},
\end{array}\right.
\end{equation}
where $\widetilde{R}_t:= \{s < t: Z_s = Z_t + 1\}$ and $\widetilde{L}_t:= \{s < t: Z_s = Z_t - 1\}$.
This (non-Markovian) random walk model was introduced by \cite{PP87} 
in the statistical physics literature. 
Based on the scaling arguments and simulations, 
it was conjectured that the walk $(Z_t, \, t \ge 0)$ is recurrent and is further super-diffusive in the sense that
$\mathbb{E}Z_t^2 \sim t^{\frac{4}{3}}$,
whereas for a simple random walk $(S_t, \, t \ge 0)$ its exploration range is $\mathbb{E}S_t^2 \sim t \ll t^{\frac{4}{3}}$.
These properties have only been proved rigorously for a simpler variant -- the edge-repelling random walk, see \cite{Davis90, Toth95}.
A counterpart to the vertex-repelling walk is the vertex-reinforced walk \cite{Pem92, Vol06} defined by $Z_{t+1} = Z_t - 1$ with probability $\frac{w(\widetilde{L}_t)}{w(\widetilde{L}_t) + w(\widetilde{R}_t)}$, and
$Z_{t+1} = Z_t + 1$ with probability $\frac{w(\widetilde{R}_t)}{w(\widetilde{L}_t) + w(\widetilde{R}_t)}$.
It is well known \cite{Tarres04, Vol06} that vertex-reinforced random walk exhibits localization at a finite number of points for some choices of $w(\cdot)$, e.g., $w(n) \sim n^{\alpha}$ with $\alpha \ge 1$.

\paragraph{Our results.} We will first show that vertex-repelling walk will never be localized or stuck at some points 
in contrast with vertex-reinforced walk (see Theorem \ref{thm:nolocal} below).
The non-localization and the (conjectured) super-diffusive properties of the vertex-repelling walk \eqref{eq:vrrw} facilitate exploration, and thus the corresponding perturbation scheme \eqref{eq:1Dreverse} makes it more likely to escape from saddle points.

We will then propose a new perturbation mechanism based on the dynamics \eqref{eq:1Dreverse}, which can be integrated into the framework of (any) perturbation-based optimization algorithms. In particular, integrating the above-mentioned mechanism into the framework of PGD and PAGD, we propose two new algorithms: perturbed gradient descent adapted to occupation time (PGDOT, Algorithm \ref{algo:PGDOTMETA}) and its accelerated version, perturbed accelerated gradient descent adapted to occupation time  (PAGDOT, Algorithm \ref{algo:PAGDOTMETA}).

We will  prove that Algorithm \ref{algo:PGDOTMETA} (resp. Algorithm \ref{algo:PAGDOTMETA}) converges to a second-order stationary point at least as fast as PGD (resp. PAGD).

\begin{algorithm}[ht]
\caption{\small{Perturbed Gradient Descent Adapted to Occupation  Time (Meta Algorithm)}}
\label{algo:PGDOTMETA}
    \begin{algorithmic} 
    \small{
      \FOR{$t = 0 ,1, \ldots$} 
        \IF{perturbation condition holds}
          \FOR{$i = 1, \ldots, d$}
            \STATE $L_t^i \leftarrow \#\{s < t: x^i_s \leq x^i_t\}$ \\
            $R_t^i \leftarrow \#\{s < t: x^i_s > x^i_t\}$
            \STATE  $x^i_t  \leftarrow \left\{ \begin{array}{rcl}
    x^i_t - \frac{r}{\sqrt{d}} \, \mbox{Unif}(0,1) & \mbox{w.p.} & p, \\ 
    x^i_t + \frac{r}{\sqrt{d}} \, \mbox{Unif}(0,1) & \mbox{w.p.} & 1-p,
    \end{array}\right.$\\
    where $p = \frac{w(R^i_t)}{w(L^i_t) + w(R^i_t)}$
          \ENDFOR
        \ENDIF        
        \STATE $\pmb{x}_{t+1} \leftarrow \pmb{x}_t - \eta  \nabla f(\pmb{x}_t)$ 
      \ENDFOR
      }
    \end{algorithmic}
\end{algorithm}

\begin{algorithm}[ht]
\caption{\small{Perturbed Accelerated Gradient Descent Adapted to Occupation Time (Meta Algorithm)}}
\label{algo:PAGDOTMETA}
    \begin{algorithmic}
    \small{
    \FOR{$t = 0, 1, \ldots, $}
    \IF{perturbation condition holds}
    
          \FOR{$i = 1, \ldots, d$}
            \STATE $L_t^i \leftarrow \#\{s < t: x^i_s \leq x^i_t\}$ \\
            $R_t^i \leftarrow \#\{s < t: x^i_s > x^i_t\}$
            \STATE $x^i_t  \leftarrow \left\{ \begin{array}{rcl}
                x^{i}_t - \frac{r}{\sqrt{d}} \mbox{Unif}(0,1) & \mbox{w.p.}
                & p, \\ 
                x^{i}_t + \frac{r}{\sqrt{d}} \mbox{Unif}(0,1) & \mbox{w.p.} &  1-p,
            \end{array}\right.$
            where $p=\frac{w(R^i_t)}{w(L^i_t) + w(R^i_t)}$
          \ENDFOR
    \ENDIF
    \STATE $\pmb{x}_{t+1} \leftarrow$ Accelerate$(\pmb{x}_t, \pmb{v}_t)$, \quad $\pmb{v}_{t+1} \leftarrow \pmb{x}_{t+1}- \pmb{x}_t$
    \ENDFOR
    }
    \end{algorithmic}
\end{algorithm}

Algorithms \ref{algo:PGDOTMETA} and \ref{algo:PAGDOTMETA} are state-dependent adaptive algorithms, perturbing GD and accelerated gradient descent (AGD) \cite{Nesterov83} non-uniformly according to the history of states.


We will finally corroborate our theoretical analysis  by experimental results.
In particular, we will demonstrate that Algorithms \ref{algo:PGDOTMETA} and \ref{algo:PAGDOTMETA} escape saddle points faster than not only their counterparts, PGD and PAGD, but also several momentum methods such as Adam, AMSGrad, and RMSProp in training multilayer perceptrons (MLPs) on some well-studied datasets such as MNIST \cite{LBBH98} and CIFAR-10 \cite{KH09}.

{\bf Notations:} Below we collect the notations that will be used throughout this paper.
For $S$ a finite set, let $\# S$ denote the number of elements in $S$. 
For $D$ as a domain, let $\mbox{Unif}(D)$ be the uniform distribution on $D$, 
e.g., $\mbox{Unif}(0,1)$ is the uniform distribution on $[0,1]$.
For a function $f: \mathbb{R}^d \to \mathbb{R}$, 
let $\nabla f$ and $\nabla^2 f$ denote its gradient and Hessian, 
and $f^{\star}: = \min_{\pmb{x} \in \mathbb{R}^d} f(\pmb{x})$ denote its global minimum.
For $\pmb{A}$ a square matrix, let $\lambda_{\min}(\pmb{A})$ be its minimum eigenvalue.

The notation $||\cdot||$ is used for both the Euclidean norm of a vector and the spectral norm of a matrix.
For $\pmb{x} = (x^1, \ldots, x^d)$ and $r > 0$, 
let $B^d({\pmb x}, r): = \{\pmb{y}: ||\pmb{y} - \pmb{x} || \le r\}$ be the $d$-dimensional ball centered at $\pmb{x}$ with radius $r$, 
and $C^d({\pmb x}, r): = \{\pmb{y}: |y^i - x^i| \le r \mbox{ for } 1 \le i \le d\}$ be the $d$-dimensional hypercube centered at $\pmb{x}$ with distance $r$ to each of its surfaces.
We use the symbol $O(\cdot)$ to hide only absolute constants which do not depend on any problem parameter.

The rest of the paper is organized as follows.
Section \ref{s2}  provides background on the continuous optimization and recalls some existing results.
Section \ref{s3} presents the main results.
Section \ref{s4} contains numerical experiments to corroborate our analysis.
Section \ref{s5} concludes.
\section{Background and Existing Results}
\label{s2}

\subsection{Results of GD}

We consider non-convex optimization (convex optimization results are recalled in Appendix B).
In this case, it is generally difficult to find the global minima. 
A popular approach is to consider the first-order stationary points instead.

\begin{definition} 
\label{def:firststat}
Let $f: \mathbb{R}^d \to \mathbb{R}$ be a differentiable function. 
We say that 
$(i)$ $\pmb{x}$ is a first-order stationary point of $f$ if $\nabla f(x) =0$;
$(ii)$ $\pmb{x}$ is an $\epsilon$-first-order stationary point of $f$ if $|| \nabla f(x)|| \le \epsilon$.
\end{definition}

We say that a differentiable function $f : \mathbb{R}^d \to \mathbb{R}$ is $\ell$-gradient Lipschitz if 
$||\nabla f (\pmb{x}_1) - \nabla f (\pmb{x}_2)|| \le \ell || \pmb{x}_1 - \pmb{x}_2||$ for all \,$\pmb{x}_1, \pmb{x}_2 \in \mathbb{R}^d$.
For gradient Lipschitz functions, GD converges to the first-order stationary points, which is quantified by the following theorem from \cite{Nesterov04}[Section 1.2.3].

\begin{theorem} 
\label{thm:GDnonconvex}
Assume that $f: \mathbb{R}^d \to \mathbb{R}$ is $\ell$-gradient Lipschitz.
For any $\epsilon >0$, if we run GD with step size $\eta = \ell^{-1}$,
then the number of iterations to find an $\epsilon$-first-order stationary point is 
$\frac{\ell (f(\pmb{x}_0) - f^{\star})}{\epsilon^2}.$
\end{theorem}

Note that in Theorem \ref{thm:GDnonconvex}, the time complexity of GD is independent of the dimension $d$.
For a non-convex function, a first-order stationary point can be either a local minimum, a saddle point, or a local maximum. 
The following definition is taken from \cite{JG17}[Definition 4].

\begin{definition} 
\label{def:saddle}
Let $f: \mathbb{R}^d \to \mathbb{R}$ be a differentiable function. 
We say that 
$(i)$ $\pmb{x}$ is a local minimum if $\pmb{x}$ is a first-order stationary point, and $f(\pmb{x}) \le f(\pmb{y})$ for all $\pmb{y}$ in some neighborhood of $\pmb{x}$;
$(ii)$ $\pmb{x}$ is a saddle point if $\pmb{x}$ is a first-order stationary point but not a local minimum.
Assume further that $f$ is twice differentiable.
We say a saddle point $\pmb{x}$ is strict if $\lambda_{\min}(\nabla^2f(\pmb{x})) < 0$.
\end{definition}

For a twice differentiable function $f$, note that $\lambda_{\min}(\nabla^2 f(\pmb{x})) \le 0$ for any saddle point $\pmb{x}$. 
So by assuming a saddle point $\pmb{x}$ to be strict, we rule out the case $\lambda_{\min}(\nabla^2 f(\pmb{x})) = 0$.
The next subsection will review two perturbation-based algorithms that allow jumping out of strict saddle points.

\subsection{Results of PGD and PAGD}
\label{sc:22}
One drawback of GD in non-convex optimization is that it may get stuck at saddle points. 
\cite{JG17} and \cite{JNM18} proposed PGD and PAGD, respectively, to escape saddle points, which we review here.
To proceed further, we need some vocabulary regarding the Hessian of the function $f$.

\begin{definition}
\label{def:secondstat}
A twice differentiable function $f: \mathbb{R}^d \to \mathbb{R}$ is $\rho$-Hessian Lipschitz if
$|| \nabla^2 f(\pmb{x}_1) - \nabla^2 f(\pmb{x}_2) || \le \rho || \pmb{x}_1 - \pmb{x}_2 ||$ for all \,$\pmb{x}_1, \pmb{x}_2 \in \mathbb{R}^d$. 
Furthermore, we say that
$(i)$ $\pmb{x}$ is a second-order stationary point of $f$ if $\nabla f(\pmb{x}) = 0$ and $\lambda_{\min} (\nabla^2 f(\pmb{x})) \ge 0$;
$(ii)$ $\pmb{x}$ is a $\epsilon$-second-order stationary point of $f$ if 
$||\nabla f(\pmb{x}) || \le \epsilon$ and $\lambda_{\min} (\nabla^2 f(\pmb{x})) \ge - \sqrt{\rho \epsilon}$.
\end{definition}

To simplify the presentation, assume that all saddle points are strict (Definition \ref{def:saddle}). 
In this situation, all second-order stationary points are local minima.
The basic idea of these two algorithms is as follows.
Imagine that we are currently at an iterate $\pmb{x}_t$ which is not an $\epsilon$-second-order stationary point. 
There are two scenarios:
$(i)$ The gradient $|| \nabla f(\pmb{x}_t)||$ is large  and a usual iteration of GD or AGD is enough;
$(ii)$ The gradient $|| \nabla f(\pmb{x}_t)||$ is small but $\lambda_{\min}(\nabla^2 f(\pmb{x}_t)) \le - \sqrt{\rho \epsilon}$ (large negative).
So $\pmb{x}_t$ is around a saddle point, and a perturbation $\xi$ is needed to escape from the saddle region: $\widetilde{\pmb{x}}_{t} = \pmb{x}_t + \xi$.

The main result for PGD, Theorem 3 in \cite{JG17}, and for PAGD, Theorem 3 in \cite{JNM18}, are stated below showing that the time complexity of these two algorithms are almost dimension-free (with a log factor).
\begin{theorem} \cite{JG17}
\label{thm:PGD}
Assume that $f: \mathbb{R}^d \to \mathbb{R}$ is $\ell$-gradient Lipschitz and $\rho$-Hessian Lipschitz.
Then there exists $c_{\max} > 0$ such that for any $\delta > 0$, $\epsilon \le \ell^2/\rho$, $\Delta_f \geq f(\pmb{x}_0)-f^*$, and $c \le c_{\max}$, 
PGD outputs an $\epsilon$-second-order stationary point with probability $1 - \delta$, terminating within the following number of iterations:
$$
O\left(\frac{\ell(f(\pmb{x}_0) - f^{\star})}{\epsilon^2} \log^4 \left(\frac{d \ell \Delta_f}{\epsilon^2 \delta} \right)\right).
$$
\end{theorem}

Compared with Theorem \ref{thm:GDnonconvex}, PGD takes almost the same order of time to find a second-order stationary point as GD does to find a first-order stationary point.

\begin{theorem} \cite{JNM18}
\label{thm:PAGD}
Assume that $f: \mathbb{R}^d \to \mathbb{R}$ is $\ell$-gradient Lipschitz and $\rho$-Hessian Lipschitz.
Then there exists an absolute constant $c_{\max} > 0$ such that for any $\delta > 0$, $\epsilon \le \ell^2/\rho$, $\Delta_f \geq f(\pmb{x}_0)-f^*$, and $c \ge c_{\max}$, with probability $1-\delta$, one of the iterates $\pmb{x}_t$ of PAGD will be an $\epsilon$-second-order stationry point in the following number of iterations:
$$
O\left(\frac{\ell^{1/2}\rho^{1/4}(f(\pmb{x}_0) - f^{\star})}{\epsilon^{7/4}} \log^6 \left(\frac{d \ell \Delta_f}{\rho\epsilon \delta} \right)\right).
$$
\end{theorem}


\section{Main Results}
\label{s3}
In this section, we first prove the non-localization property of the vertex-repelling random walk. Then, we formalize the idea of perturbations adapted to occupation time and provide the full version of PGDOT and PAGDOT in Algorithms \ref{algo:PGDOT} and \ref{algo:PAGDOT}, respectively. 
Our main results show that these algorithms converge rapidly to second-order stationary points.

\subsection{Non-Localization Property of Vertex-Repelling Random Walk}

The following theorem suggests that the new perturbation mechanism helps perturbation-based algorithms to avoid getting stuck at saddle points, as  the dynamics of  vertex-repelling random walk prescribed in \eqref{eq:1Dreverse} does not localize.

\begin{theorem}
\label{thm:nolocal}
Let $\{Z_t, \, t = 0, 1, \ldots\}$ be the vertex-repelling random walk defined by \eqref{eq:vrrw}, where
$w: \{0, 1, \ldots\} \rightarrow (0, \infty)$ is an increasing function such that $w(n) \rightarrow \infty$ as $n \rightarrow \infty$.
Then 
\begin{equation*}
\mathbb{P}\left(\exists t_0>0, \, k \le \ell: Z_t \in \{k, \ldots, \ell\} \mbox{ for all } t \ge t_0   \right) = 0.
\end{equation*}
\end{theorem}

The proof of this theorem is given in Appendix C.

\subsection{Perturbed Gradient Descent Adapted to Occupation Time}
PGD adds a uniform random perturbation when stuck at saddle points.
From the discussion in the introduction, it is more reasonable to perturb with non-uniform noise whose distribution depends on the occupation times.
Recall that  $w: \{0,1, \ldots\} \to (0, \infty)$ is an increasing weight function on the nonnegative integers.
The following algorithm adapts PGD to random perturbation depending on the occupation dynamics.
We follow the parameter setting as in \cite{JG17}.
Our algorithm performs GD with step size $\eta$ and gets a perturbation of amplitude $\frac{r}{\sqrt{d}}$ near saddle points at most once every $t_{\tiny \mbox{thres}}$ iterations. 
The threshold $t_{\tiny \mbox{thres}}$ ensures that the dynamics of the algorithm is mostly GD. 
The threshold $g_{\tiny \mbox{thres}}$ determines if a perturbation is needed, and the threshold $f_{\tiny \mbox{thres}}$ decides when the algorithm terminates.

\begin{algorithm}
  \caption{\small{Perturbed Gradient Descent Adapted to Occupation Time: PGDOT($\pmb{x}_0, \ell, \rho, \epsilon, c, \delta, \Delta_f)$}}
  \label{algo:PGDOT}
  \begin{algorithmic}
  \small{
      \STATE $\chi \leftarrow 3 \max\left\{\log(\frac{d \ell \Delta_f}{c \epsilon^2 \delta}), 4\right\}$, \,$\eta \leftarrow \frac{c}{\ell}$, \,$r \leftarrow \frac{\epsilon\sqrt{c}}{\chi^2 \ell}$
      \STATE $g_{\tiny \mbox{thres}} \leftarrow \frac{\epsilon \sqrt{c}}{\chi^2}$, \,$f_{\tiny \mbox{thres}} \leftarrow \frac{c}{\chi^3} \sqrt{\frac{\epsilon^3}{\rho}}$, \,$t_{\tiny \mbox{thres}} \leftarrow \frac{\chi \ell}{c^2 \sqrt{\rho \epsilon}}$    
      \STATE $t_{\tiny \mbox{noise}} \leftarrow - t_{\tiny \mbox{thres}} - 1$                
      \FOR{$t = 0 ,1, \ldots$} 
        \IF{$|| \nabla f(\pmb{x}_t)|| \le g_{\tiny \mbox{thres}}$ and $t - t_{\tiny \mbox{noise}} > t_{\tiny \mbox{thres}}$}
          \STATE $\widetilde{\pmb{x}}_{t} \leftarrow \pmb{x}_t$, \,$t_{\tiny \mbox{noise}} \leftarrow t$ 
          \FOR{$i = 1, \ldots, d$}
            \STATE $L_t^i \leftarrow \#\{s < t: x^i_s \leq x^i_t\}$ \\
            $R_t^i \leftarrow \#\{s < t: x^i_s > x^i_t\}$
            \STATE $x^i_t  \leftarrow \left\{ \begin{array}{rcl}
\widetilde{x}^{i}_t - \frac{r}{\sqrt{d}} \mbox{Unif}(0,1) & \mbox{w.p.}
& p, \\ 
\widetilde{x}^{i}_t + \frac{r}{\sqrt{d}} \mbox{Unif}(0,1) & \mbox{w.p.} &  1-p,
\end{array}\right.$
where $p=\frac{w(R^i_t)}{w(L^i_t) + w(R^i_t)}$
          \ENDFOR
        \ENDIF
        \IF{$t - t_{\tiny \mbox{noise}} = t_{\tiny \mbox{thres}}$ and $f(\pmb{x}_t) - f(\widetilde{\pmb{x}}_{t_{\tiny \mbox{noise}}}) > - f_{\tiny \mbox{thres}}$}
        \STATE {\bf return} $\widetilde{\pmb{x}}_{t_{\tiny \mbox{noise}}}$
        \ENDIF
        \STATE $\pmb{x}_{t+1} \leftarrow \pmb{x}_t - \eta  \nabla f(\pmb{x}_t)$ 
      \ENDFOR
      }
  \end{algorithmic}
\end{algorithm}

The next theorem gives the convergence rate of Algorithm \ref{algo:PGDOT}: PGDOT finds a second-order stationary point in the same number of iterations (up to a constant factor) as PGD does.

\begin{theorem} 
\label{thm:PGDOT}
Assume that $f: \mathbb{R}^d \to \mathbb{R}$ is $\ell$-gradient Lipschitz and $\rho$-Hessian Lipschitz.
Then there exists $c_{\max} > 0$ such that for any $\delta > 0$, $\epsilon \le \ell^2/\rho$, $\Delta_f \geq f(\pmb{x}_0)-f^*$, and $c \le c_{\max}$, PGDOT (Algorithm \ref{algo:PGDOT}) outputs an $\epsilon$-second-order stationary point with probability $1 - \delta$ terminating within the following  number of iterations:
$$
O\left(\frac{\ell(f(\pmb{x}_0) - f^{\star})}{\epsilon^2} \log^4 \left(\frac{d \ell \Delta_f}{\epsilon^2 \delta} \right)\right).
$$
\end{theorem}

The proof of Theorem \ref{thm:PGDOT} is based on a geometric characterization of saddle points -- thin pancake property \cite{JG17}. 
In Appendix D, we will discuss this property, and show how it is used to prove Theorem \ref{thm:PGDOT}.

\subsection{Perturbed Accelerated Gradient Descent Adapted to Occupation Time}
Similar to the way we combined our perturbation mechanism with PGD, we can  adapt PAGD to this mechanism as well resulting in the accelerated version of PGDOT (Algorithm \ref{algo:PAGDOT}). We follow the parameter setting as in \cite{JNM18}.

\begin{algorithm}[ht]
\caption{\small{Perturbed Accelerated Gradient Descent Adapted to Occupation Time: PAGDOT($\pmb{x}_0, \eta, \theta, \gamma, s, r, \utime$)}}
\label{algo:PAGDOT}
    \begin{algorithmic}
    \small{
    \STATE $\chi \leftarrow  \max\left\{\log(\frac{d \ell \Delta_f}{\rho \epsilon \delta}), 1\right\}$, \, $\kappa \leftarrow \frac{\ell}{\sqrt{\rho\epsilon}}$, \,$\eta \leftarrow \frac{1}{4\ell}$
    \STATE $\theta \leftarrow \frac{1}{4\sqrt{\kappa}}$, \, $\gamma \leftarrow \frac{\theta^2}{\eta}$, \, $s \leftarrow \frac{\gamma}{4\rho}$, \,
    $r \leftarrow \frac{\eta\epsilon}{\chi^5 c^8}$, \, $\utime \leftarrow \chi c \sqrt{\kappa}$
    \STATE $\pmb{v}_0 \leftarrow 0$
    \FOR{$t = 0, 1, \ldots, $}
    \IF{$\|{\nabla f(\pmb{x}_t)}\| \le \epsilon$ and \emph{no perturbation in last $\utime$ steps}}
    
    \STATE $\widetilde{\pmb{x}}_{t} \leftarrow \pmb{x}_t$ 
          \FOR{$i = 1, \ldots, d$}
            \STATE $L_t^i \leftarrow \#\{s < t: x^i_s \leq x^i_t\}$ \\
            $R_t^i \leftarrow \#\{s < t: x^i_s > x^i_t\}$
            \STATE $x^i_t  \leftarrow \left\{ \begin{array}{rcl}
                \widetilde{x}^{i}_t - \frac{r}{\sqrt{d}} \mbox{Unif}(0,1) & \mbox{w.p.}
                & p, \\ 
                \widetilde{x}^{i}_t + \frac{r}{\sqrt{d}} \mbox{Unif}(0,1) & \mbox{w.p.} &  1-p,
            \end{array}\right.$
            where $p=\frac{w(R^i_t)}{w(L^i_t) + w(R^i_t)}$
          \ENDFOR
    \ENDIF
    \STATE $\pmb{y}_{t} \leftarrow \pmb{x}_{t} + (1-\theta) \pmb{v}_t$
    \STATE $\pmb{x}_{t+1} \leftarrow \pmb{y}_t - \eta \nabla f (\pmb{y}_t)$
    \STATE $\pmb{v}_{t+1} \leftarrow \pmb{x}_{t+1} - \pmb{x}_t $
    \IF{$f(\pmb{x}_t) \le  f(\pmb{y}_t) + \langle \nabla f(\pmb{y}_t), \pmb{x}_t - \pmb{y}_t \rangle - \frac{\gamma}{2} \|{\pmb{x}_t - \pmb{y}_t}\|^2$}
    \STATE $(\pmb{x}_{t+1}, \pmb{v}_{t+1}) \leftarrow $ \text{NCE($\pmb{x}_t, \pmb{v}_t, s$)}
    \ENDIF
    \ENDFOR
    }
    \end{algorithmic}
\end{algorithm}
Algorithm \ref{algo:PAGDOT}, similar to PAGD, enjoys a feature enabling it to reset the momentum and decide whether to exploit the negative curvature when the function becomes ``too convex" (see Algorithm \ref{algo:NCE}).

\begin{algorithm}[ht]
\caption{\small{Negative Curvature Exploitation: NCE($\pmb{x}_t, \pmb{v}_t, s$)}}
\label{algo:NCE}
    \begin{algorithmic}
    \small{
    \IF{$\|{\pmb{v}_t}\| \geq s$}
     \STATE $\pmb{x}_{t+1} \leftarrow \pmb{x}_t$
    \ELSE
    \STATE $\delta = s.\pmb{v}_t / \|{\pmb{v}_t}\|$ 
    \STATE $\pmb{x}_{t+1} \leftarrow \arg \min_{\pmb{x} \in \{\pmb{x}_t+\delta, \pmb{x}_t-\delta\}} f(\pmb{x})$
    \ENDIF
    \STATE \textbf{return} $(\pmb{x}_{t+1}, 0)$
    }
    \end{algorithmic}
\end{algorithm}

The next theorem gives the convergence rate of Algorithm \ref{algo:PAGDOT}: PAGDOT finds a second-order stationary point in the same number of iterations (up to a constant factor) as PAGD does, and therefore achieves a faster convergence rate than PGD and PGDOT. The proof of Theorem \ref{thm:PAGDOT} is similar to that of Theorem \ref{thm:PGDOT} (see Appendix D).

\begin{theorem} 
\label{thm:PAGDOT}
Assume that $f: \mathbb{R}^d \to \mathbb{R}$ is $\ell$-gradient Lipschitz and $\rho$-Hessian Lipschitz.
Then there exists an absolute constant $c_{\max} > 0$ such that for any $\delta > 0$, $\epsilon \le \ell^2/\rho$, $\Delta_f \geq f(\pmb{x}_0)-f^*$, and $c \ge c_{\max}$, one of the iterates $\pmb{x}_t$ of PAGDOT (Algorithm \ref{algo:PAGDOT}) will be an $\epsilon$-second-order stationry point in the following number of iterations, with probability $1-\delta$:
$$
O\left(\frac{\ell^{1/2}\rho^{1/4}(f(\pmb{x}_0) - f^{\star})}{\epsilon^{7/4}} \log^6 \left(\frac{d \ell \Delta_f}{\rho\epsilon \delta} \right)\right).
$$
\end{theorem}

It is worth mentioning that Algorithms \ref{algo:PGDOT} and \ref{algo:PAGDOT} share some spirit with simulated annealing and GD with momentum methods such as the heavy ball method \cite{Polyak164}.
In simulated annealing, the perturbation is  {\it time-adapted} while the perturbation in Algorithms \ref{algo:PGDOT} and \ref{algo:PAGDOT} is {\it state-adapted} (to the history of states).
In the heavy ball method, a momentum term, which is a function of the current and previous states, is explicitly added to control the  oscillations and accelerate in low curvatures along the direction close to momentum.
In Algorithms \ref{algo:PGDOT} and \ref{algo:PAGDOT}, however, no explicit momentum term is added. Instead, the perturbation is adapted to the history of states providing the current state with an explicit direction.

\section{Empirical Results}
\label{s4}

This section presents empirical results to corroborate the theoretical analysis presented in the previous section. Different machine learning tasks are considered including a nonlinear regression problem adapted from learning time series data, a regularized linear quadratic problem, the phase retrieval problem, and training MLPs on the MNIST and CIFAR-10 datasets. 

As shown in these examples, integrating our new perturbation mechanism into the framework of perturbation-based algorithms boosts their performance: PGDOT and PAGDOT escape saddle points or plateaus faster than their counterparts. Specifically, example 4 shows that in training MLPs on the MNIST and CIFAR-10 datasets, PGDOT and PAGDOT are robust against different initialization and manage to escape saddle points efficiently; in contrast, PGD and PAGD as well as other popular algorithms such as stochastic gradient descent (SGD), Adam, AMSGrad, and RMSProp fail to do so. In addition, PAGDOT converges faster than PGDOT in all these examples, which is in line with the theoretical results in Theorems \ref{thm:PGDOT} and \ref{thm:PAGDOT}.

In these experiments, we use $L^i_t(h): = \# \{ t-t_{\tiny \mbox{count}} \leq s < t: x^i_t - h \leq
x^i_s \leq x^i_t \}$ 
and $R^i_t(h): = \# \{t - t_{\tiny \mbox{count}} \leq s < t: x^i_t  < x^i_s \leq x^i_t + h \}$ instead of $L^i_t$ and $R^i_t$ in Algorithms \ref{algo:PGDOT} and \ref{algo:PAGDOT}.
Here $h$ is a hyperparameter characterizing the occupation time over a small interval. 
$t_{\tiny \mbox{count}}$ is another hyperparameter prescribing how long one should keep track of the history of $\pmb{x}_t$ in order to approximate the occupation time with a constant memory cost. We choose the weight function in Algorithms \ref{algo:PGDOT} and \ref{algo:PAGDOT} as $w(n)=1+n^{5}$.
All other hyperparameters used in the numerical examples are reported in Appendix E.


\paragraph{Example 1}
Given $N\in \mathbb{Z}^+, L\in \mathbb{R}^+$,  define a function $\tilde{f}: \mathbb{R}^+ \to \mathbb{R}^+$ as
\small
\begin{equation*}
    \tilde{f}(r)= 
\begin{cases}
    r^3, & r\in[0, \frac12L), \\
    (r-nL)^3 + \frac14nL^3, & r\in[a(n), b(n)),  1\leq n \leq N, \\
    (r-NL)^3 + \frac14NL^3, & r\in [NL+\frac12L, \infty),
\end{cases}
\end{equation*}
\normalsize
where $a(n) = nL-\frac12L$ and $b(n) = nL+\frac12L$.
For $\pmb{x}=(x_1,\dots,x_d)\in \mathbb{R}^d$, we define 
 $   f(\pmb{x}) = \tilde{f}\left(\frac1d{\sum_{i=1}^d x_i^2}\right).$
Figure \ref{fig:example1} gives the visualization of the case $N=4, L=1$ and also the training curves of $f$ given by 5 different algorithms when $d=4$. The initial values are all the same, and all the algorithms except for GD are run 3 times considering the randomness of perturbations. 
We can see that while GD gets stuck at the saddle points, all the rest of the algorithms escape from them. Moreover, the new algorithms PGDOT and PAGDOT outperform their counterparts. 

\begin{figure}[ht]
\centering
\begin{subfigure}
    \centering
    \includegraphics[width=0.3\columnwidth]{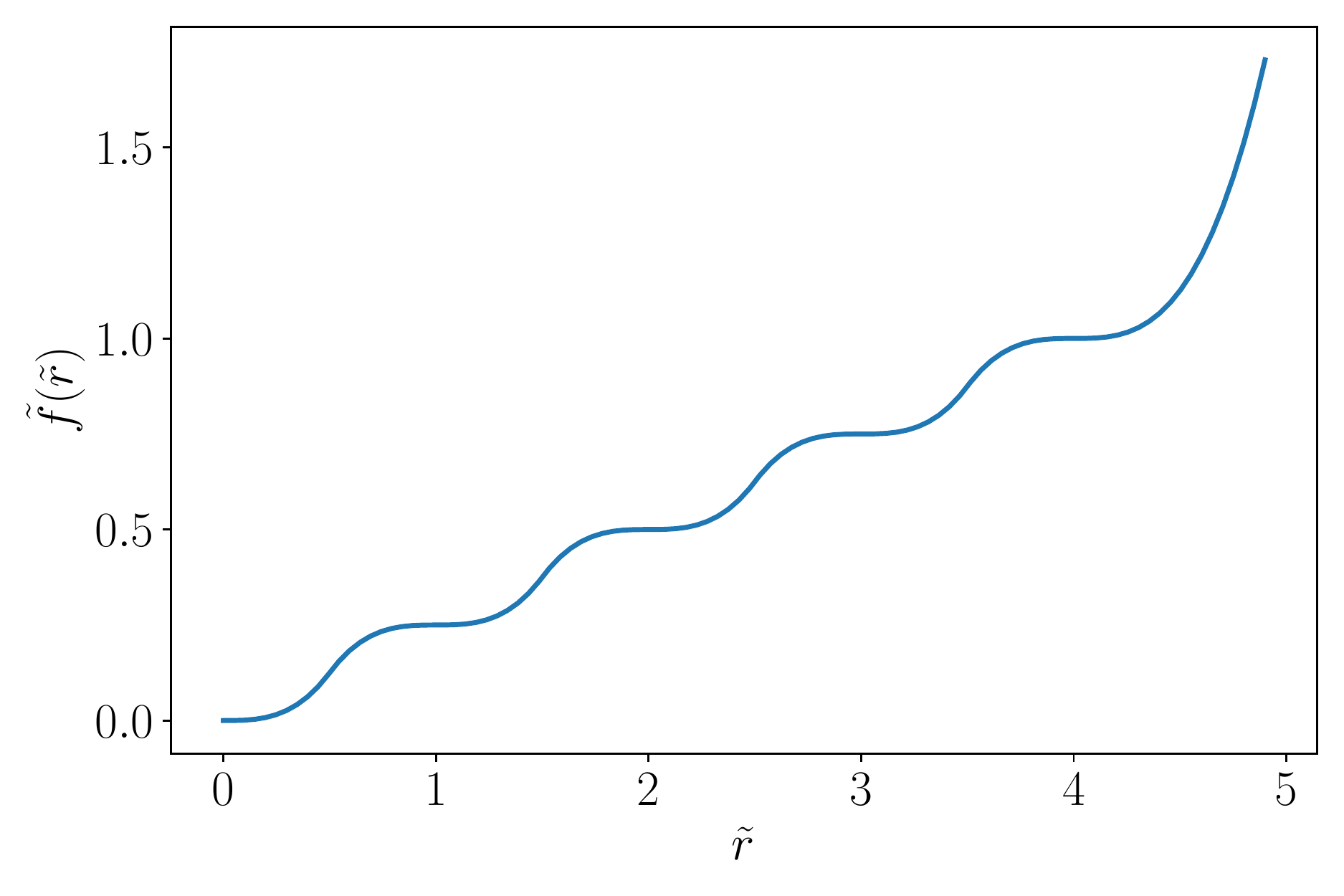}
\end{subfigure}
\vspace{-0.1in}
\begin{subfigure}
    \centering
    \includegraphics[width=0.33\columnwidth]{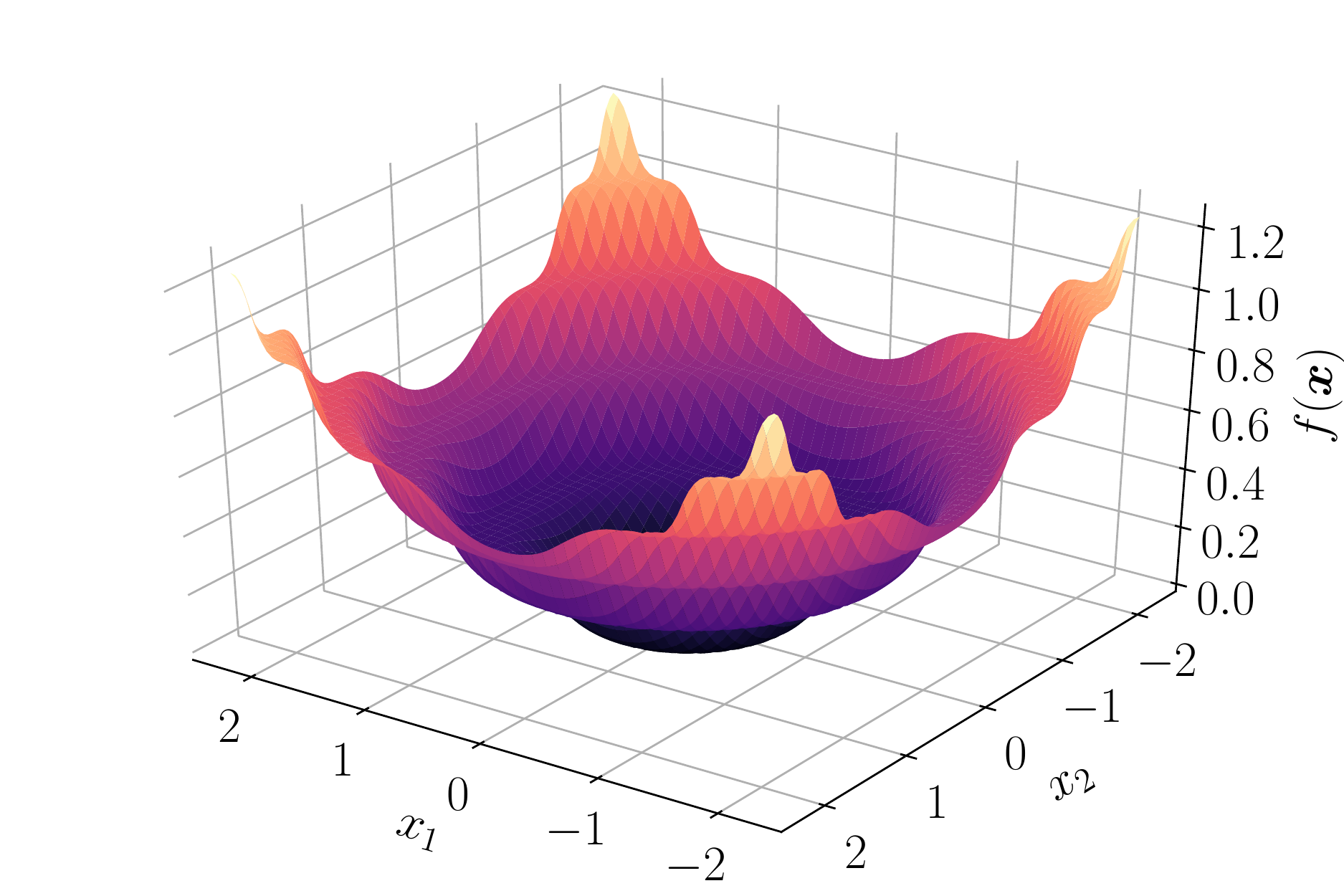}
\end{subfigure} 
\begin{subfigure}
    \centering
    \includegraphics[width=0.4\columnwidth]{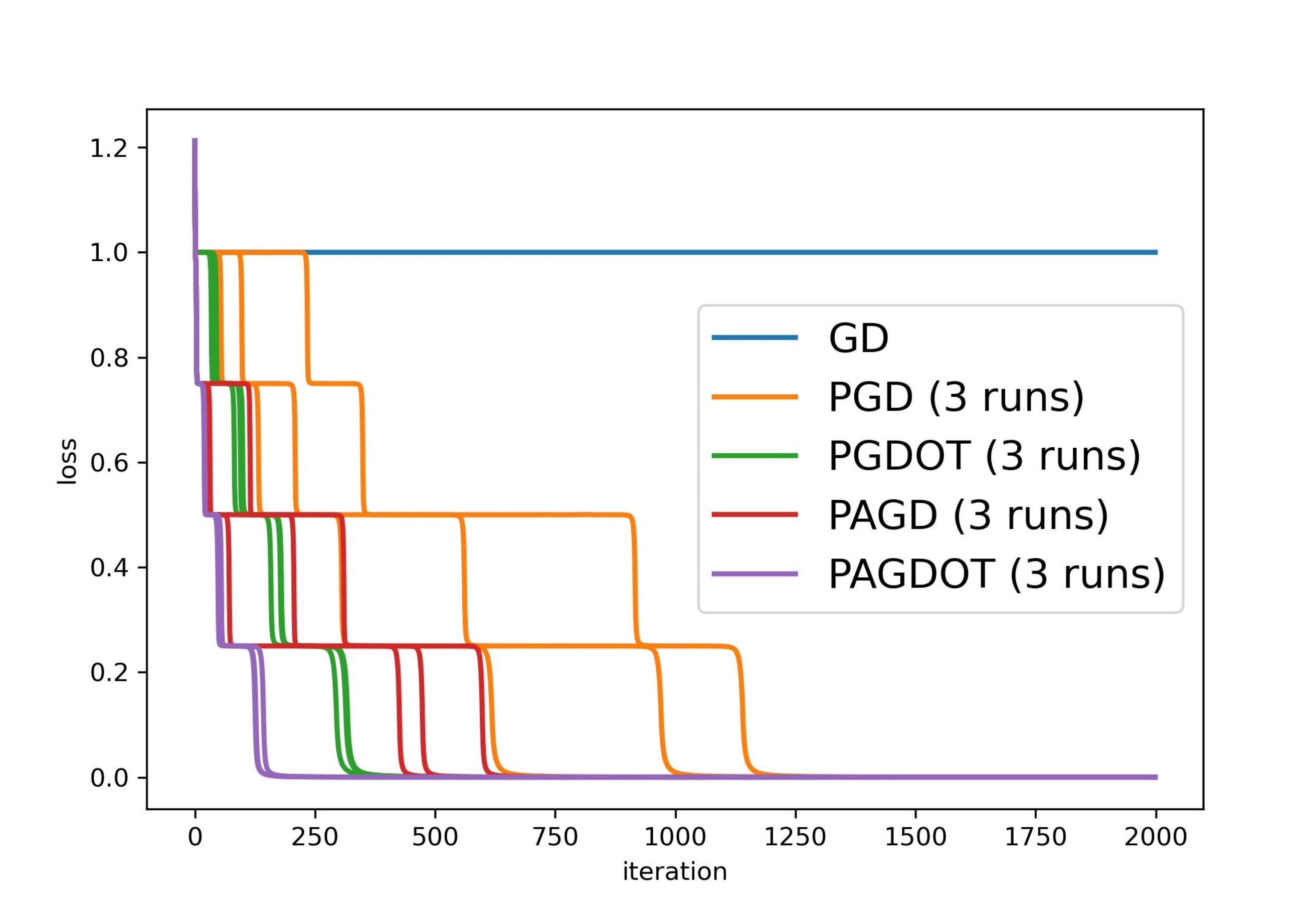}
\end{subfigure}
\vspace{-0.1in}
\caption{Graph of $\tilde{f}$ (top left) and the landscape of $f(\pmb{x})$ (top right) with $\pmb{x} \in \mathbb{R}^2$ in the case of $N=4, L=1$, and the performance of different algorithms (bottom) when $N=4, L=1, d=4$  in Example 1.}
\vspace{-0.08in}
\label{fig:example1}
\end{figure}

\paragraph{Example 2}
We consider a nonlinear regression problem, adapted from learning time series data with a continuous dynamical system~\cite{LHEL21}. The loss function is defined as
$
    f(\pmb{x}) = \frac1N\sum_{i=1}^N(\hat{y}(s_i;\pmb{x}) - y^*(s_i))^2,
$
where $\{s_i\}_{i=1}^N$ are $N$ sample points, $y^*(s)$ is the target function, and $\hat{y}(s)$ is the function to fit with the form
$
    \hat{y}(s;\pmb{x}) = \sum_{m=1}^M(a_{m}\cos{(\lambda_ms)} + b_{m}\sin{(\lambda_ms)})e^{w_ms}.
$
Here $\pmb{x}=\{a_m,b_m,\lambda_m,w_m\}_{m=1}^M$ and the optimization problem is non-convex. We assume 
$y^*(s)=\mathrm{Ai}(\omega [s - s_0])$, where $\omega=3.2, s_0=3.0$, and $\mathrm{Ai}(s)$ is the Airy function of the first kind, given by the improper integral
$\mathrm{Ai}(s) = \frac{1}{\pi} \int_{0}^{\infty} \cos \left(\frac{u^{3}}{3}+s u\right) \mathrm{d} u$.

For the specific regression model, we assume $M=4$ and use $N=50$ data points with $s_i=i/10, i=0,\dots,49$. 
Figure \ref{fig:example2} shows the target function and the fitted function obtained by PGDOT. Also, the learning curves of 5 different algorithms are plotted. 
Again, PGDOT and PAGDOT escape the saddle point faster than GD and outperform  PGD and PAGD, respectively.

\begin{figure}[ht]
\centering
\hspace{-0.33in}
\vspace{-0.05in}
\includegraphics[width=0.39\columnwidth]{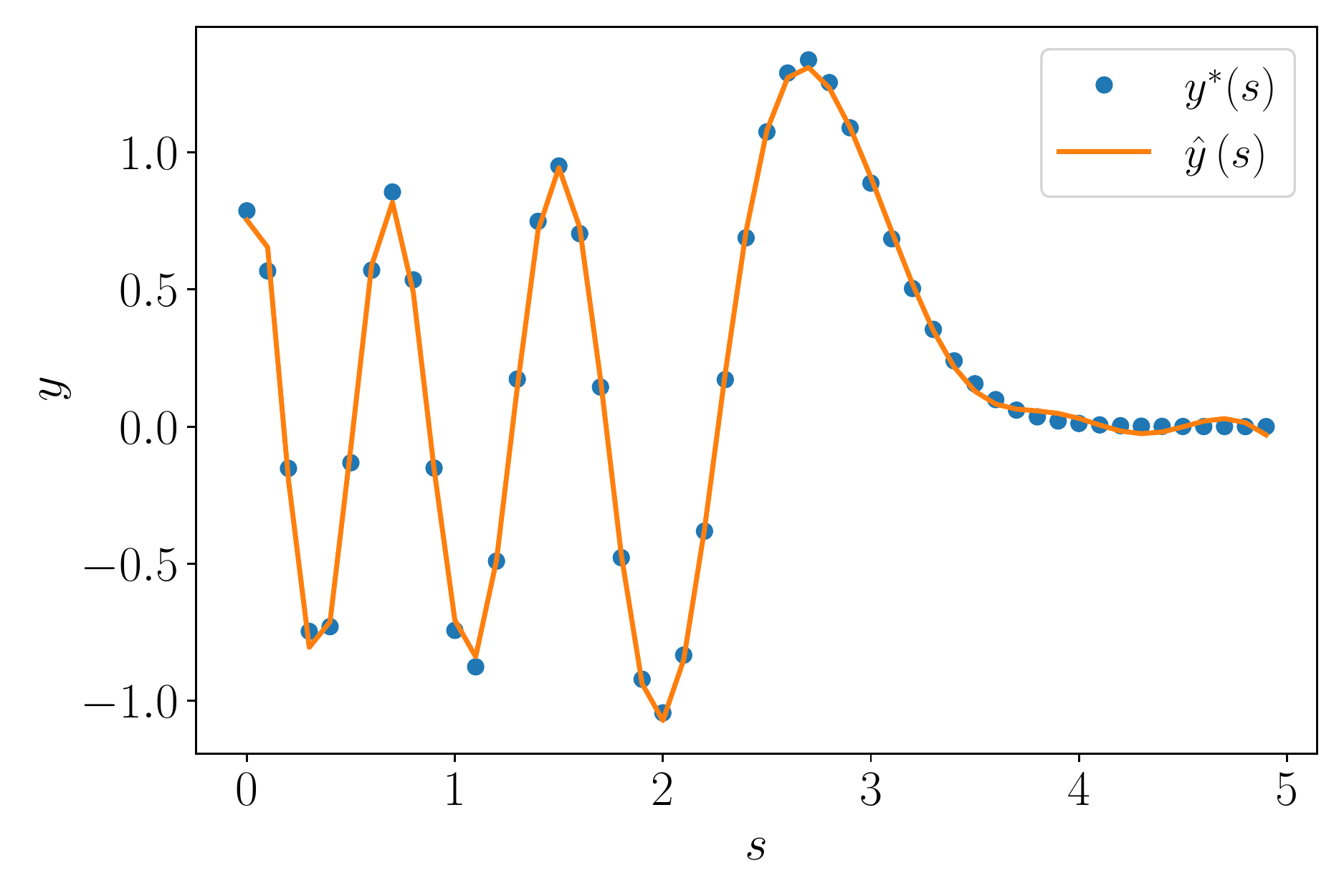}
\includegraphics[width=0.41\columnwidth]{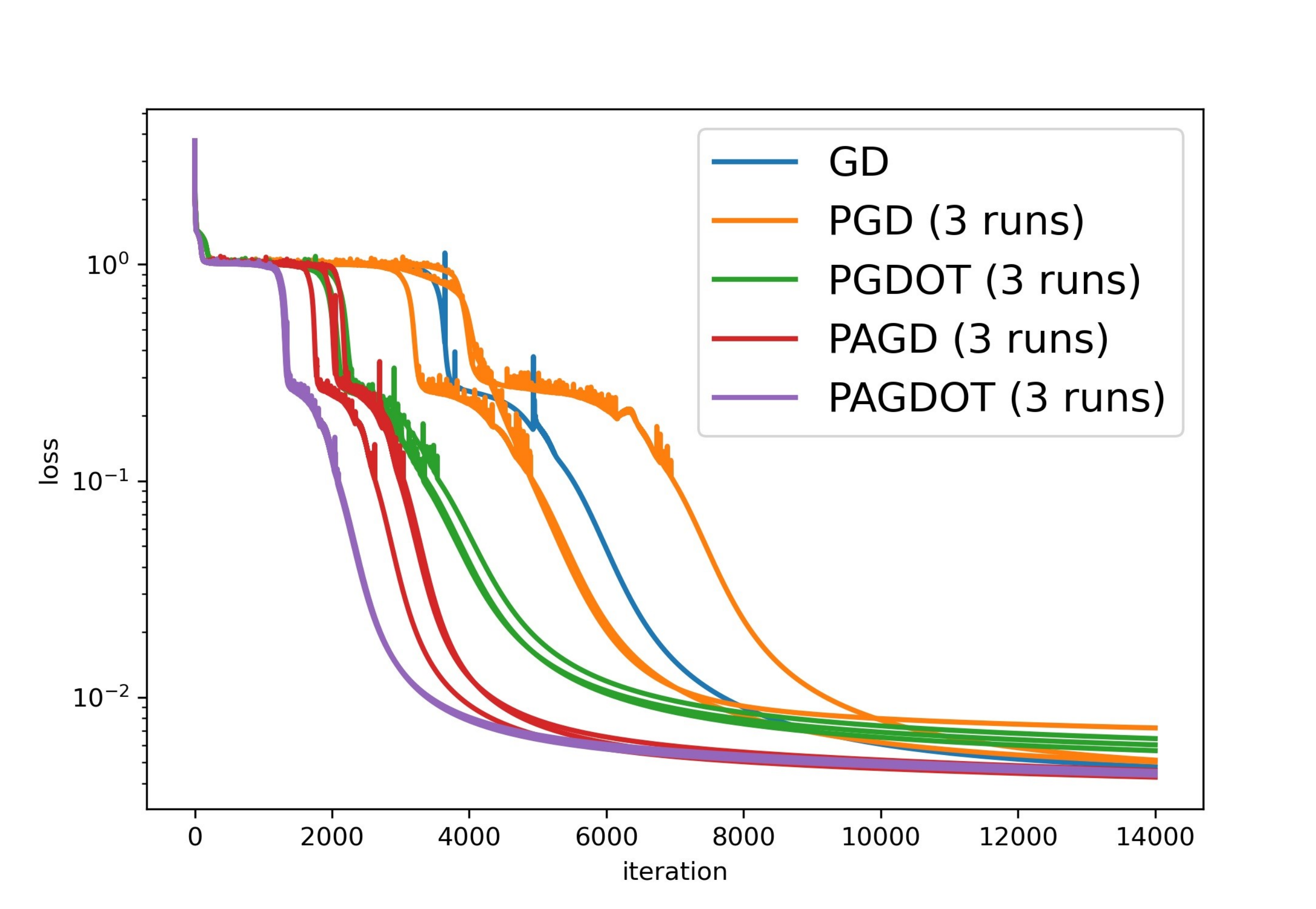}
\vspace{-0.1in}
\caption{The target function $y^*(t)$ and fitted function $\hat{y}(t)$  obtained by PGDOT (left), and the performance of different algorithms (right) in Example 2.}
\vspace{-0.1in}
\label{fig:example2}
\end{figure}

\paragraph{Example 3}
The next two non-convex optimization problems are taken from \cite{WLA19}.
The first problem is a regularized linear-quadratic problem \cite{RZ18}, whose loss function is
\begin{equation*}
    f_1(\pmb{x}) = \frac{1}{N} \sum_{i=1}^N(\frac{1}{2}\pmb{x}^TH\pmb{x}+\pmb{b}_i^T\pmb{x}+||\pmb{x}||_{10}^{10}),
\end{equation*} 
where we take $N = 10$, $H=\text{diag}([1, -0.1])$ and $\pmb{b}_i$'s instances of  $\mathcal{N}(0,\text{diag}([0.1, 0.001]))$. 
The second problem is the phase retrieval problem \cite{CE13} with loss function
\begin{equation*}
    f_2(\pmb{x}) = \frac{1}{N}\sum_{i=1}^N((\pmb{a}_i^T\pmb{x})^2 -(\pmb{a}_i^T\pmb{x}^*)^2)^2,
\end{equation*}
where we choose $N = 200$, $\pmb{x}^*$ an instance of $\mathcal{N}(0,I_d/d)$ and $\pmb{a}_i$'s instances of $\mathcal{N}(0,I_d)$ with $d=10$.

We initialize the regularized linear-quadratic problem with $\pmb{x}_0 = 0$, and the phase retrieval problem with $\pmb{x}_0$ sampled from $\mathcal{N}(0,I_d/(10000d))$.
Figure \ref{fig:example3} presents the learning curves of 5 different algorithms.
In both problems,  all other algorithms escape saddle points faster than GD, with PGDOT and PAGDOT outperforming their counterparts.

\begin{figure}[ht]
\centering
\hspace{-0.09in}
\vspace{-0.05in}
\includegraphics[width=0.4\columnwidth]{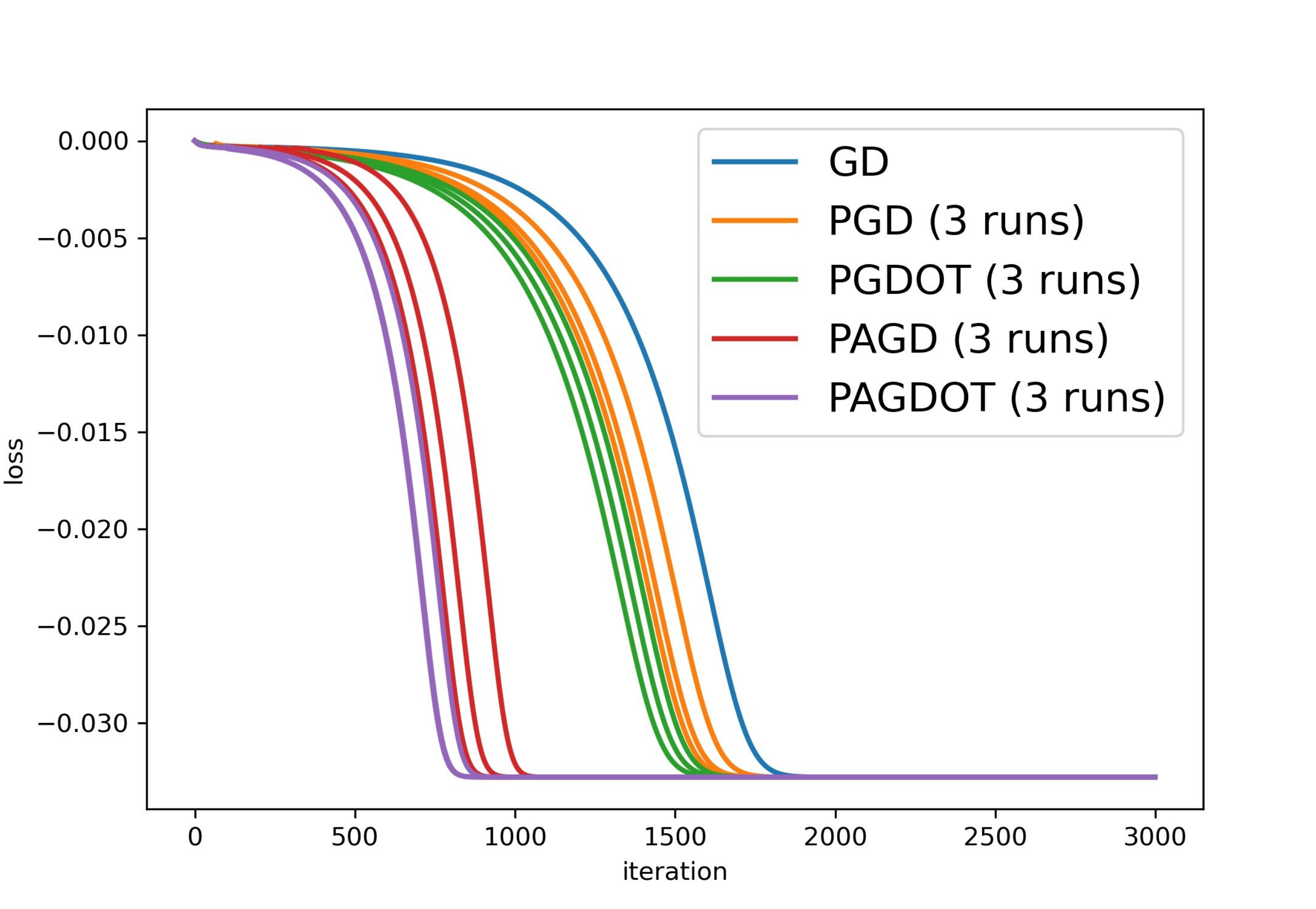}
\includegraphics[width=0.39\columnwidth]{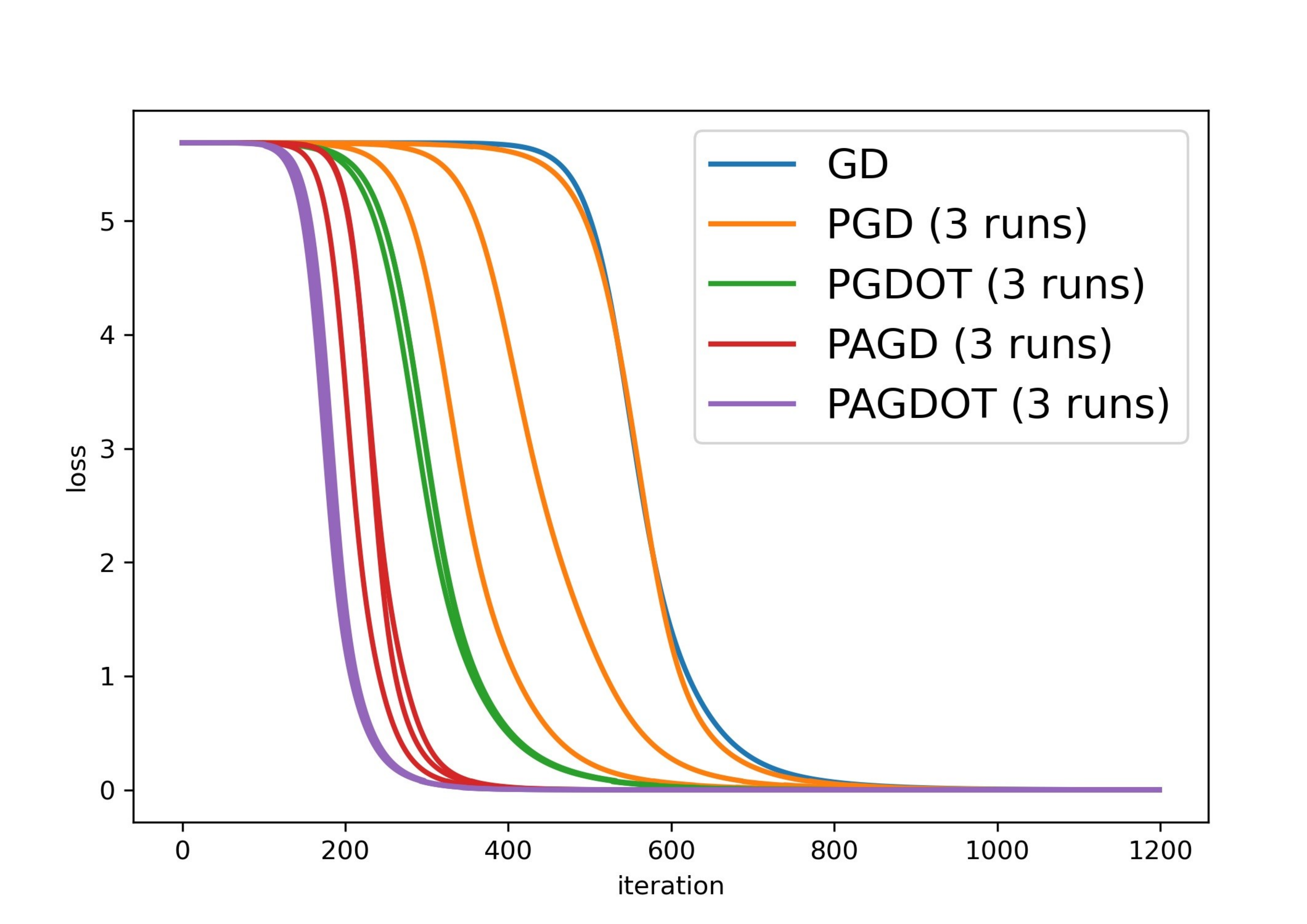}
\vspace{-0.1in}
\caption{The performance of different algorithms for $f_1$ (left) and $f_2$ (right) in Example 3.}
\vspace{-0.2in}
\label{fig:example3}
\end{figure}

\paragraph{Example 4}
\cite{dauphin2014identifying} observed that in training simple MLPs (MLPs with only one hidden layer) on the MNIST and CIFAR-10 datasets, SGD might get stuck at saddle points.
Moreover, as demonstrated in \cite{swirszcz2016local}, Adam also gets stuck and performs poorly when a simple MLP with specific initialization is trained on the MNIST dataset. Inspired by \cite{swirszcz2016local}, we conduct two sets of experiments on the MNIST and CIFAR-10 datasets, in which we train several simple MLPs using the mini-batch version of our proposed algorithms as well as the mini-batch version of other popular alternatives such as SGD, Adam, AMSGrad, and RMSProp. For both of the datasets, the batch size is set to 128 and all the images are downsized to be of size $10 \times 10$.

In the first set of experiments, we train several simple MLPs whose weights and biases are initialized with $\mathcal{N}(0, 0.01)$ on the aforementioned datasets. The top two rows in Figure \ref{fig:example4} show the training curves of SGD, Adam, PGD, PGDOT, PAGD, and PAGDOT. Note that $n\_hidden$ is the number of neurons in the hidden layer of simple MLP. Observe that all algorithms manage to escape saddle points in all the cases with the exception of Adam, which fails on the CIFAR-10 dataset. 

For the second set of experiments, we consider several simple MLPs whose weights and biases are initialized with $\mathcal{N}(-1, 0.01)$.
The bottom two rows in Figure \ref{fig:example4} show the training curves of different algorithms. For both of the datasets, while SGD and Adam are stuck at the saddle points, the new algorithms PGDOT and PAGDOT escape the saddle points and significantly outperform their counterparts. Note that for the CIFAR-10 dataset, the training curves of SGD, Adam, PGD, and PAGD are almost identical.

\begin{figure*}[ht]
    \centering
    \includegraphics[width=0.65\linewidth]{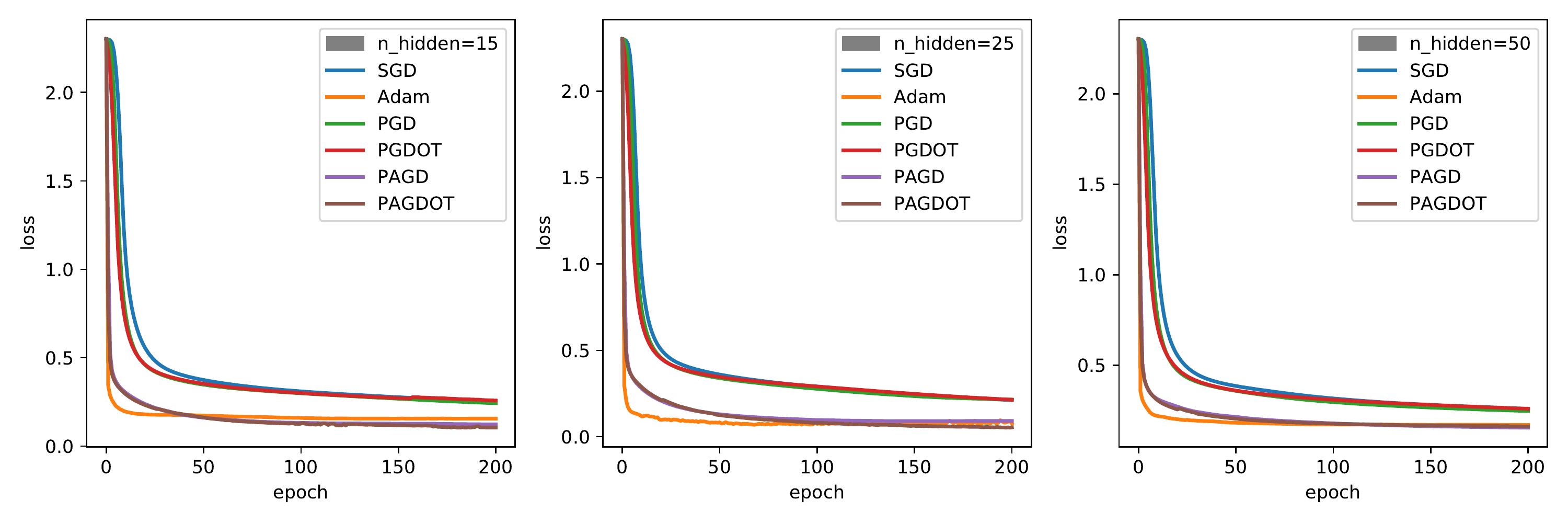}
     \includegraphics[width=0.65\linewidth]{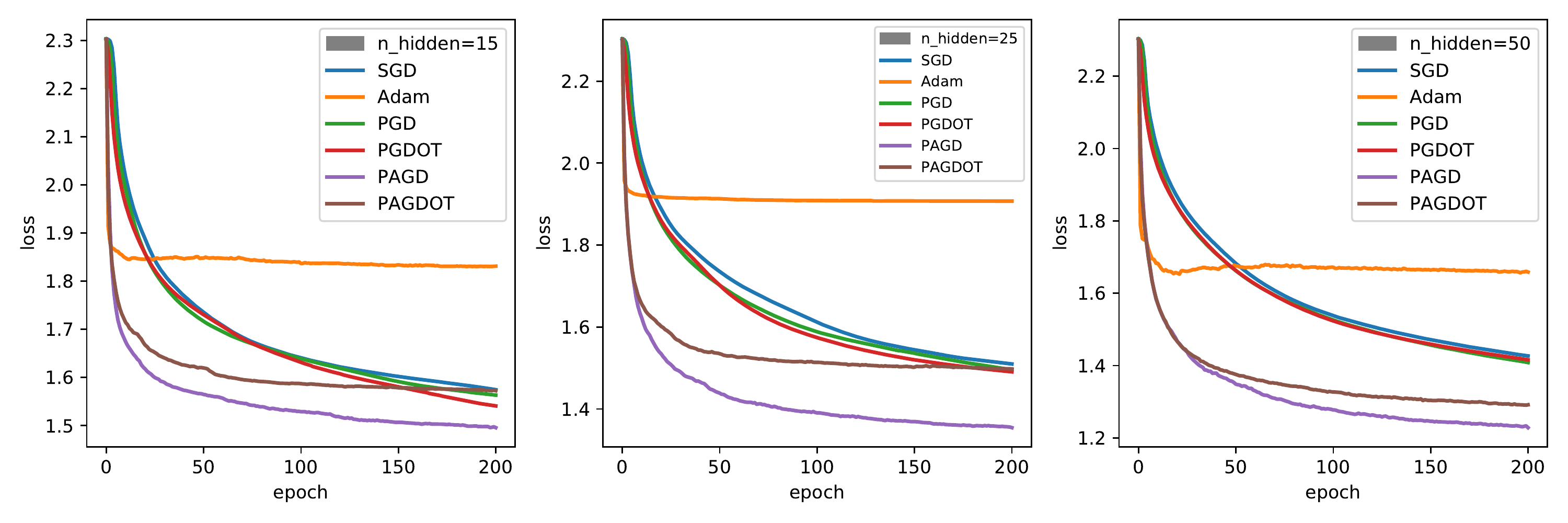}
     \includegraphics[width=0.65\linewidth]{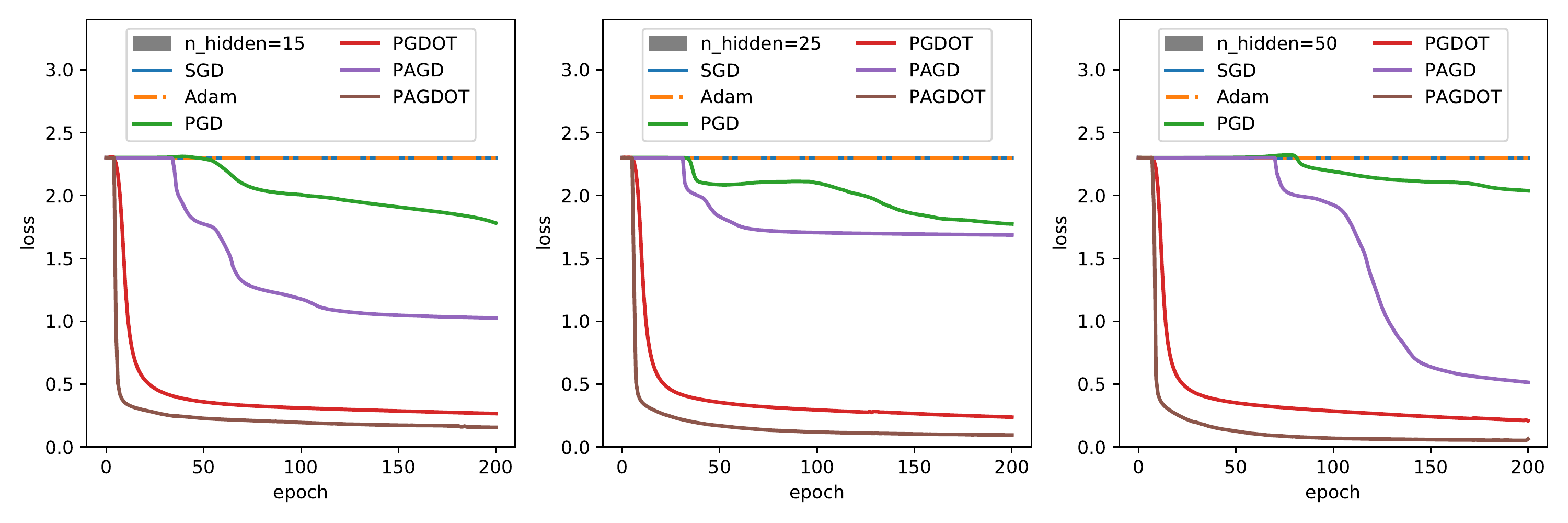}
    \includegraphics[width=0.65\linewidth]{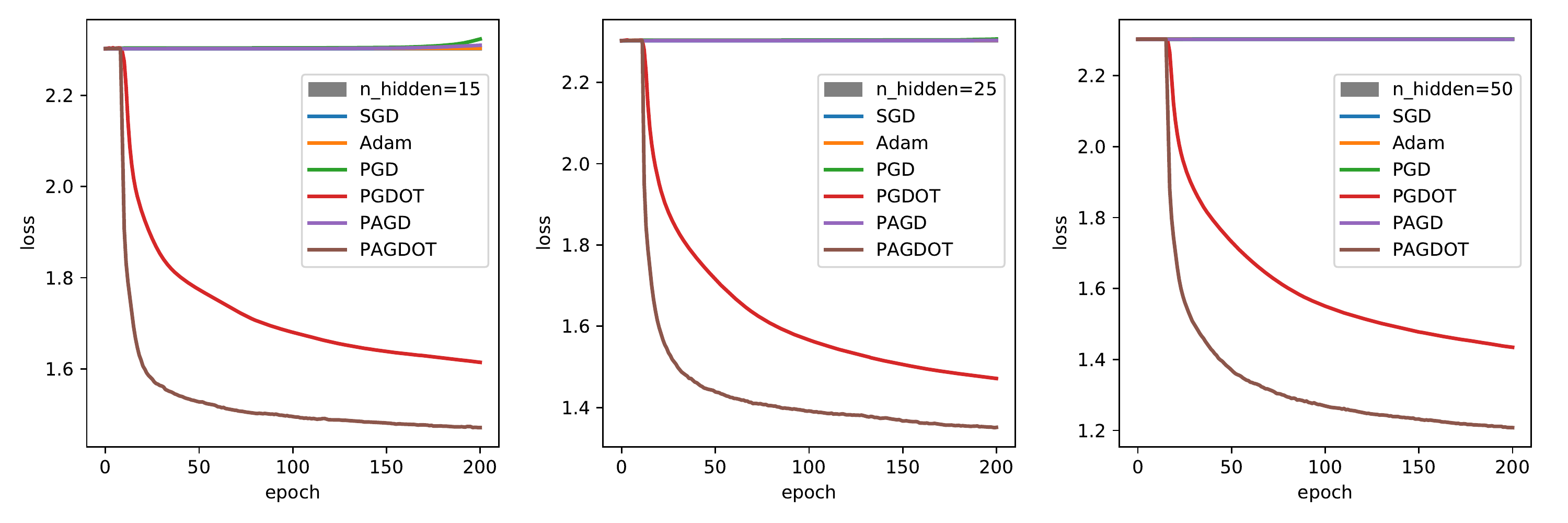}
    \vspace{-0.1in}
    \caption{Training simple MLPs with two different initialization on the MNIST and CIFAR-10 datasets. First row: $\mathcal{N}(0, 0.01)$ initialization, MNIST; Second row: $\mathcal{N}(0, 0.01)$ initialization, CIFAR-10; Third row: $\mathcal{N}(-1, 0.01)$ initialization, MNIST; Fourth row: $\mathcal{N}(-1, 0.01)$ initialization, CIFAR-10.}
    \vspace{-0.12in}
    \label{fig:example4}
\end{figure*}

It is also worth mentioning that other variants of Adam such as AMSGrad \cite{RKK19} and RMSProp also fail to escape the saddle points in the training process (see Figure \ref{fig:example4_apdx} in Appendix F). Moreover, comparing these results with that of the first set of experiments (Figure \ref{fig:example4}, top two rows) we conclude that the new perturbation mechanism helps the algorithms to be robust against different initialization.



\section{Conclusion}
\label{s5}

In this paper, we develop a new perturbation mechanism in which the perturbations are adapted to the history of states via the notion of occupation time. This mechanism is integrated into the framework of PGD and PAGD resulting in two new algorithms: PGDOT and PAGDOT. 
We prove that PGDOT and PAGDOT converge rapidly to second-order stationary points, which is corroborated by empirical studies ranging from 
time series analysis and  the phase retrieval problem to neural networks.

\bibliography{example_paper}
\bibliographystyle{icml2022}



\newpage
\appendix
\onecolumn

\section{Monotone convergence of gradient descent}

Here we prove a property of gradient descent applied to a function $f: \mathbb{R} \to \mathbb{R}$, as mentioned in the introduction. 
This property of gradient descent supports the use of our proposed perturbation mechanism.

\begin{proposition}
Let $f \in \mathcal{C}^2(\mathbb{R})$. 
Assume that we start gradient descent at some arbitrary point $x_0$, and the corresponding iterates $\{x_n\}_{n\geq0}$ converge to the point $x_s$ with $f''(x_s) \neq 0$. Then, if $f$ is $\ell$-smooth and the step size is less than $\frac{1}{\ell}$, the sequence $\{x_n\}_{n\geq 0}$ converges monotonically to $x_s$.
\end{proposition}

In order to prove this proposition, we break it down into two lemmas.

\begin{lemma}
\label{lem:B1}
Let $f \in \mathcal{C}^2(\mathbb{R})$. Assume that we start gradient descent at some arbitrary point $x_0$, and the corresponding iterates $\{x_n\}_{n\geq0}$ converge to the point $x_s$ with $f''(x_s) \neq 0$. Then, if $f$ is $\ell$-smooth and the step size is less than $\frac{1}{\ell}$, there exists $M > 0$ such that the sequence $\{x_n\}_{n\geq M}$ converges monotonically to $x_s$.
\end{lemma}
\begin{proof}
Note that for $n\geq 0$,
$x_{n+1} = x_n - \eta f'(x_n)$,
where $0<\eta<\frac{1}{\ell}$ is the step size. Also, it is easy to show that $f'(x_s)=0$.
Assume that at some point $x_n\geq x_s$. Then, since $|f'(x_n)-f'(x_s)|=|f'(x_n)|\leq \ell|x_n-x_s|$,  we have 
\begin{align*}
    x_s \leq x_n -\frac{1}{\ell}|f'(x_n)| & \leq x_n-\eta |f'(x_n)| \\
    &\leq  x_n - \eta f'(x_n) = x_{n+1}.
\end{align*}
Similarly, if $x_n \leq x_s$, then we get $x_{n+1} \leq x_s$. This implies that the sequence $\{x_n\}_{n\geq 0}$ is entirely either on the left hand side of $x_s$ or on its right hand side (including $x_s$). 

Without loss of generality, assume that the entire sequence of iterations lies on the right hand side of $x_s$.
If at some iteration, $x_m = x_s$, then since $f'(x_s)=0$, $x_n= x_s$ for $n\geq m$, which yields the desired result. 
So we can assume that $x_n \neq x_s $ for all $n\geq 0$. Using a similar argument, we can also assume that $f'(x_n) \neq 0$ for all $n\geq 0$.
Suppose by contradiction that there is no such $M$ as described in the lemma. 
Then there exist infinitely many $n$ such that $x_n<x_{n+1}$ implying that for infinitely many $n$, $f'(x_n)<0$. 
Since $\lim\limits_{n \rightarrow \infty} x_n = x_s$ and the entire sequence is on the right hands side of $x_s$, 
we also have infinitely many $n$ such that $f'(x_n)>0$. 
Combining these results, one can construct a strictly decreasing sub-sequence $\{y_n\}_{n\geq 0}$ of the iterations such that $\lim\limits_{n \rightarrow \infty} y_n = x_s$, $f'(y_{2m})>0$, and $f'(y_{2m+1})<0$ for all $m\geq 0$.  Since $f'$ is continuous, there exists $y_{2m+1}<z_m<y_{2m}$ such that $f'(z_m)=0$, for each $m\geq 0$. It is easy to see that $\{z_n\}_{n\geq 0}$ is also  strictly decreasing and $\lim\limits_{n \rightarrow \infty} z_n = x_s$. 
Note that since $f''$ is continuous, by the mean value theorem, one can find a sequence $\{t_n\}_{n\geq 0}$ such that for each $n\geq 0$, $z_{n+1} < t_n < z_n$ and $f''(t_n)=0$. Since $\{z_n\}_{n\geq 0}$ converges to $x_s$, then so does $\{t_n\}_{n \geq 0}$. But this implies that $f''(x_s) = \lim\limits_{n \rightarrow \infty} f''(t_n) = 0$ contradicting with the fact that $f''(x_s) \neq 0$.
\end{proof}

\begin{lemma}
Given the setting in Lemma \ref{lem:B1}, $\{x_n\}_{n\geq 0}$ converges monotonically to $x_s$.
\end{lemma}
\begin{proof}
Without loss of generality, assume that $x_0\geq x_s$, then using what we obtained during the proof of Lemma \ref{lem:B1}, we know that the entire sequence $\{x_n\}_{n\geq 0}$ lies on the right hand side of $x_s$ ((including $x_s$). 
Let $M$ be the minimum index that satisfies the condition in Lemma \ref{lem:B1}. 
Suppose by contradiction that $M>0$. 
So $x_s < x_{M-1}<x_M$, which implies $f'(x_{M-1})<0$ considering $x_M = x_{M-1} - \eta f'(x_{M-1})$. 
Since the sequence converges to $x_s$, there should be a $k\geq 0$ such that $x_{M+k+1}< x_{M-1} < x_{M+k}$. 
Note that 
$x_{M+k+1} = x_{M+k} - \eta f'(x_{M+k})$, so
\begin{align*}
    \eta f'(x_{M+k}) 
     = x_{M+k} - x_{M+k+1} 
     > x_{M+k} - x_{M-1}.
\end{align*}
Since $f'(x_{M-1})<0$, we have $\eta \big(f'(x_{M+k})-f'(x_{M-1})\big ) > \eta f'(x_{M+k}) > x_{M+k} - x_{M-1}$.
This contradicts the fact that $\eta(f'(x_{M+k})-f'(x_{M-1})) \leq \eta \ell (x_{M+k}-x_{M-1})<x_{M+k}-x_{M-1}$. 
\end{proof}

\section{Background on convex optimization}
We provide some context of gradient descent applied to convex functions.

\begin{definition} 
\label{def:ellapha}
~
\begin{enumerate}
\item
A differentiable function $f : \mathbb{R}^d \to \mathbb{R}$ is $\ell$-gradient Lipschitz if 
$||\nabla f (\pmb{x}_1) - \nabla f (\pmb{x}_2)|| \le \ell || \pmb{x}_1 - \pmb{x}_2||$ for all \,$\pmb{x}_1, \pmb{x}_2 \in \mathbb{R}^d$. 
\item
A twice differentiable function $f : \mathbb{R}^d \to \mathbb{R}$ is $\alpha$-strongly convex if 
$\lambda_{\min}(\nabla^2 f(\pmb{x})) \ge \alpha$ for all \,$\pmb{x} \in \mathbb{R}^d$.
\end{enumerate}
\end{definition}

The gradient Lipschitz condition controls the amount of decay in each iteration, 
and the strong convexity condition guarantees that the unique stationary point is the global minimum.
The ratio $\ell/\alpha$ is often called the condition number of the function $f$.
The following theorem shows the linear convergence of gradient descent to the global minimum $\pmb{x}^{\star}$, 
see \cite{Bubeck15}[Theorem 3.10] and  \cite{Nesterov04}[Theorem 2.1.15].

\begin{theorem} \cite{Bubeck15, Nesterov04}
\label{thm:GDconvex}
Assume that $f: \mathbb{R}^d \to \mathbb{R}$ is $\ell$-gradient Lipschitz and $\alpha$-strongly convex.
For any $\epsilon >0$, if we run gradient descent with step size $\eta = \ell^{-1}$, 
then the number of iterations to be $\epsilon$-close to $\pmb{x}^{\star}$ is
$\frac{2 \ell}{\alpha} \log \left(\frac{||\pmb{x}_0 - \pmb{x}^{\star}||}{\epsilon} \right).$
\end{theorem}

\section{Proof of Theorem \ref{thm:nolocal}}
Suppose by contradiction that with positive probability, the walk is localized at some points $\{k, \ldots, \ell\}$.
We focus on the left end $k$. 
Let $\tau^k_n$ be the time at which the point $k$ is visited $n$ times. For $n$ sufficiently large, the point $k+1$ is visited approximately at least $n$ times by $\tau^k_n$.
So at time $\tau^k_n$, the walk moves from $k$ to $k+1$ with probability bounded from above by 
$C/w(n)$ for some constant $C>0$.
Consequently, the probability that the walk is localized at $\{k, \ldots, \ell\}$ is less than $\prod_{n > 0} \frac{C}{w(n)}$. 
By standard analysis, $\prod_{n > 0} \frac{C}{w(n)} =  0$ if $w(n) \rightarrow \infty$ as $n \rightarrow \infty$.
This leads to the desired result.

\section{Proof of Theorem \ref{thm:PGDOT}}
We show how the thin-pancake property of saddle points is used to prove Theorem \ref{thm:PGDOT}.
Recall that an $\epsilon$-second-order stationary point is a point with a small gradient, and where the Hessian does not have a large negative eigenvalue.
Let us put down the basic idea in Section \ref{sc:22} with the parameters in Algorithm \ref{algo:PGDOT} (PGDOT).
If we are currently at an iterate $\pmb{x}_t$ which is not an $\epsilon$-second-order stationary point, there are two cases:
(1)
The gradient is large: $|| \nabla f(\pmb{x_t})|| \ge g_{\tiny \mbox{thres}}$;
(2)
$\pmb{x}_t$ is close to a saddle point: $|| \nabla f(\pmb{x_t})|| \le g_{\tiny \mbox{thres}}$ 
and $\lambda_{\min} ( \nabla^2 f(\pmb{x_t})) \le -\sqrt{\rho \epsilon}$.
The case $(1)$ is easy to deal with by the following elementary lemma. 
\begin{lemma}
\label{lem:largegrad}
Assume that $f: \mathbb{R}^d \to \mathbb{R}$ is $\ell$-gradient Lipschitz. 
Then for GD with step size $\eta < \ell^{-1}$, we have
$f(\pmb{x}_{t+1}) - f(\pmb{x}_t) \le -\frac{\eta}{2} ||\nabla f(\pmb{x}_t) ||^2$.
\end{lemma}

The case $(2)$ is more subtle, and the following lemma gives the decay of the function value after a random perturbation described in Algorithm \ref{algo:PGDOT} (PGDOT).

\begin{lemma}
\label{lem:nearsaddle}
Assume that $f: \mathbb{R}^d \to \mathbb{R}$ is $\ell$-gradient Lipschitz and $\rho$-Hessian Lipschitz.
If $|| \nabla f(\pmb{x_t})|| \le g_{\tiny \mbox{thres}}$ 
and $\lambda_{\min} ( \nabla^2 f(\pmb{x_t})) \le -\sqrt{\rho \epsilon}$, 
then adding one perturbation step as in Algorithm \ref{algo:PGDOT} followed by $t_{\tiny \mbox{thres}}$ steps of GD with step size $\eta$, 
we have $f(\pmb{x}_{t + t_{\tiny \mbox{thres}}}) - f(\pmb{x}_t) \le - f_{\tiny \mbox{thres}}$ with probability at least $1 - \frac{d \ell}{\sqrt{\rho \epsilon}} e^{-\chi}$.
\end{lemma}

\cite{JG17} proved Lemma \ref{lem:nearsaddle} for PGD, 
and used it together with Lemma \ref{lem:largegrad} to prove Theorem \ref{thm:PGD}.
We will use  the same argument, with Lemmas \ref{lem:largegrad} and \ref{lem:nearsaddle}, leading to Theorem \ref{thm:PGDOT} for PGDOT.

Now, let us explain how to prove Lemma \ref{lem:nearsaddle} via a purely geometric property of saddle points. 
Consider a point $\widetilde{\pmb{x}}$ satisfying the condition 
$|| \nabla f(\widetilde{\pmb{x}})|| \le g_{\tiny \mbox{thres}}$ 
and $\lambda_{\min} ( \nabla^2 f(\widetilde{\pmb{x}})) \le -\sqrt{\rho \epsilon}$.
After adding the perturbation in Algorithm \ref{algo:PGDOT}, the resulting vector can be viewed as a distribution over the cube $C^{(d)}(\widetilde{\pmb{x}}, r/\sqrt{d})$.
Similar as in \cite{JG17}, we call $C^{(d)}(\widetilde{\pmb{x}}, r/\sqrt{d})$ the perturbation cube
which is divided into two regions: 
(1)
escape region $\chi_{\tiny \mbox{escape}}$ which consists of all points $\pmb{x} \in C^{(d)}(\widetilde{\pmb{x}}, r/\sqrt{d})$ whose function value decreases by at least $f_{\tiny \mbox{thres}}$ after $t_{\tiny \mbox{thres}}$ steps;
(2)
stuck region $\chi_{\tiny \mbox{stuck}}$ which is the complement of $\chi_{\tiny \mbox{escape}}$ in $C^{(d)}(\widetilde{\pmb{x}}, r/\sqrt{d})$.
The key idea is that the stuck region $\chi_{\tiny \mbox{stuck}}$ looks like a non-flat thin pancake, which has a very small volume compared to that of $C^{(d)}(\widetilde{\pmb{x}}, r/\sqrt{d})$.
This claim can be formalized by the following lemma, which is a direct corollary of \cite{JG17}[Lemma 11] as $C^{(d)}(\widetilde{\pmb{x}}, r/\sqrt{d}) \subseteq B^d(\widetilde{\pmb{x}}, r)$:

\begin{lemma} 
\label{lem:pancake}
Assume that $\widetilde{\pmb{x}}$ satisfies $|| \nabla f(\widetilde{\pmb{x}})|| \le g_{\tiny \mbox{thres}}$ 
and $\lambda_{\min} ( \nabla^2 f(\widetilde{\pmb{x}})) \le -\sqrt{\rho \epsilon}$.
Let $\pmb{e}_1$ be the smallest eigendirction of $\nabla^2 f(\widetilde{\pmb{x}})$. 
For any $\delta < 1/3$ and any $\pmb{u}, \pmb{v} \in C^{(d)}(\widetilde{\pmb{x}}, r/\sqrt{d})$,
if $\pmb{u} - \pmb{v} = \mu r \pmb{e}_1$ and $\mu \ge \delta/(2 \sqrt{d})$,
then at least one of $\pmb{u}$ and $\pmb{v}$ is not in the stuck region $\chi_{\tiny \mbox{stuck}}$.
\end{lemma}

To prove Lemma \ref{lem:nearsaddle},  it suffices to check that 
$\mathbb{P}(\chi_{\tiny \mbox{stuck}}) \le C \delta$ for some $C > 0$.
This criterion is general for any (random) perturbation. 
Let $\mathcal{O}_1, \ldots, \mathcal{O}_{2^d}$ be the orthants centered at $\widetilde{\pmb{x}}$; that is, the space $\mathbb{R}^d$ is divided into $2^d$ subspaces according to the coordinate signs of $\cdot - \widetilde{\pmb{x}}$.
The symbol $\sgn(\mathcal{O}_i) \in \{-1,1\}^d$ denotes the coordinate signs of $\pmb{y} - \widetilde{\pmb{x}}$ for any $\pmb{y} \in \mathcal{O}_i$.
For $1 \le i \le 2^d$, let
\begin{equation*}
\scalemath{0.9}{
p_i  = \prod_{\sgn(\mathcal{O}_i)_k = -1} \frac{w(R^k_t)}{w(L^k_t) + w(R^k_t)}\prod_{\sgn(\mathcal{O}_i)_k = 1} \frac{w(L^k_t)}{w(L^k_t) + w(R^k_t)}}
\end{equation*}
be the probability that the random perturbation drives $\widetilde{\pmb{x}}$ into $C^{(d)}(\widetilde{\pmb{x}}, r/\sqrt{d}) \cap \mathcal{O}_i$.
Consequently,
$\mathbb{P}(\chi_{\tiny \mbox{stuck}}) = \sum\limits_{i = 1}^{2^d} p_i \frac{\mbox{Vol}(\chi_{\tiny \mbox{stuck}} \cap \mathcal{O}_i)}{\mbox{Vol}(C^{(d)}(\widetilde{\pmb{x}}, r/\sqrt{d}) \cap \mathcal{O}_i)}$,
where $\mbox{Vol}(\cdot)$ denotes the volume of a domain. 
It is easy to see that $\mbox{Vol}(C^{(d)}(\widetilde{\pmb{x}}, r/\sqrt{d}) \cap \mathcal{O}_i) = (r/\sqrt{d})^d$. 
By Lemma \ref{lem:pancake} and the slicing volume bound \cite{Ball86},
$\mbox{Vol}(\chi_{\tiny \mbox{stuck}} \cap \mathcal{O}_i) \le \sqrt{2} (r/\sqrt{d})^{d-1} \frac{\delta r}{\sqrt{d}}$.
Therefore,
$\frac{\mbox{Vol}(\chi_{\tiny \mbox{stuck}} \cap \mathcal{O}_i)}{\mbox{Vol}(C^{(d)}(\widetilde{\pmb{x}}, r/\sqrt{d}) \cap \mathcal{O}_i)}
\le \sqrt{2} \delta$ implying that $\mathbb{P}(\chi_{\tiny \mbox{stuck}}) \le \sqrt{2} \delta$.

Note that this proof does not rely on the full history of states for $L_t$ and $R_t$. Thus, one can restrict the number of previous iterations as is done in Section \ref{s4} using the hyperparameter $t_{\text{count}}$.

\section{Hyperparameters in the numerical examples}

 Table \ref{tab_parameter} summarizes the hyperparameters used for the empirical studies in Section \ref{s4}. 

\begin{table}[ht]
\caption{Hyperparameters.}\label{tab_parameter}
\begin{center}
\begin{tabular}{@{}l|cccccccccc@{}}
\toprule
 & $d$ & \# of steps & $h$ & $t_{\tiny \mbox{count}}$ & $\eta$ & $t_{\tiny \mbox{thres}}$ &$g_{\tiny \mbox{thres}}$ & $r$ & momentum & batch size\\ \midrule
Example 1 & 4  & 2000  & 0.04  & 200 & 0.1  & 10 & 0.01 & 0.04 & 0.5 & -\\
Example 2 & 16 & 14000 & 0.04 &  200 & 0.1 & 50 & 0.1 & 0.1 & 0.5 & -\\
Example 3 (1) & 2 & 3000 & 1 &  200 & 0.01  & 50 & 0.01 & 0.01 & 0.5 & -\\
Example 3 (2) & 10 & 1200 & 1 &  200 & 0.001  & 50 & 1 & 0.01 & 0.5 & -\\
Example 4 (MNIST) & - & 200 (epochs) & $10^{12}$ &  50 & 0.01  & 10 & 0.1 & 0.5 & 0.9 & 128\\
Example 4 (CIFAR-10) & - & 200 (epochs) & $10^{12}$ &  50 & 0.01  & 10 & 0.1 & 0.5 & 0.9 & 128\\
\bottomrule
\end{tabular}
\end{center}
\end{table}
Note that the implementation of PAGDOT is almost the same as the implementation of PGDOT with one difference: instead of GD, Nesterov's AGD is deployed. Also, for the last two sets of experiments, $h$ is set to be $10^{12}$, a very large number. This way we are basically ignoring the hyperparameter $h$. 

All the experiments are conducted on either a local machine or Google Colab using a CPU. Each of the first three examples takes a few minutes to run, and each of the experiments in example 4 takes a few hours.

\section{AMSGrad and RMSProp also fail}

Figure \ref{fig:example4_apdx} confirms that besides SGD and Adam, some of their variants such as AMSGrad and RMSProp also fail in the training process when the weights and biases of the simple MLPs are initialized with $\mathcal{N}(-1, 0.01)$. Note that $n\_hidden$ is the number of neurons in the hidden layer of simple MLP.

\begin{figure}[ht]
    \centering
    \includegraphics[width=0.85\linewidth]{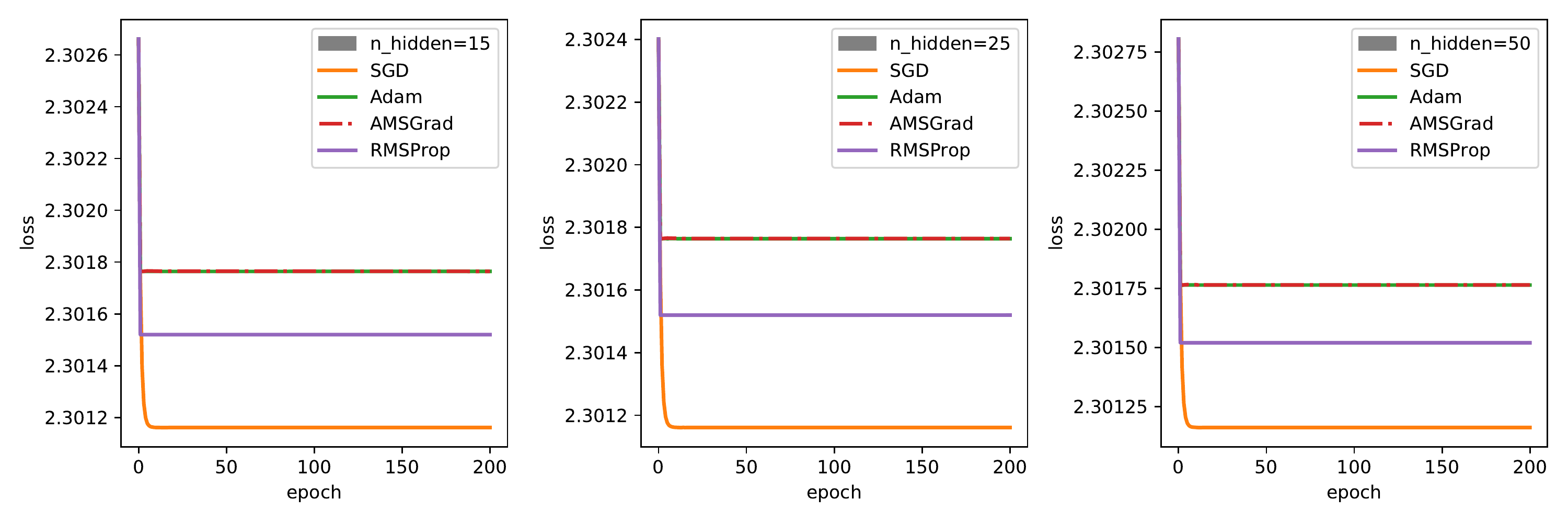}
    \includegraphics[width=0.85\linewidth]{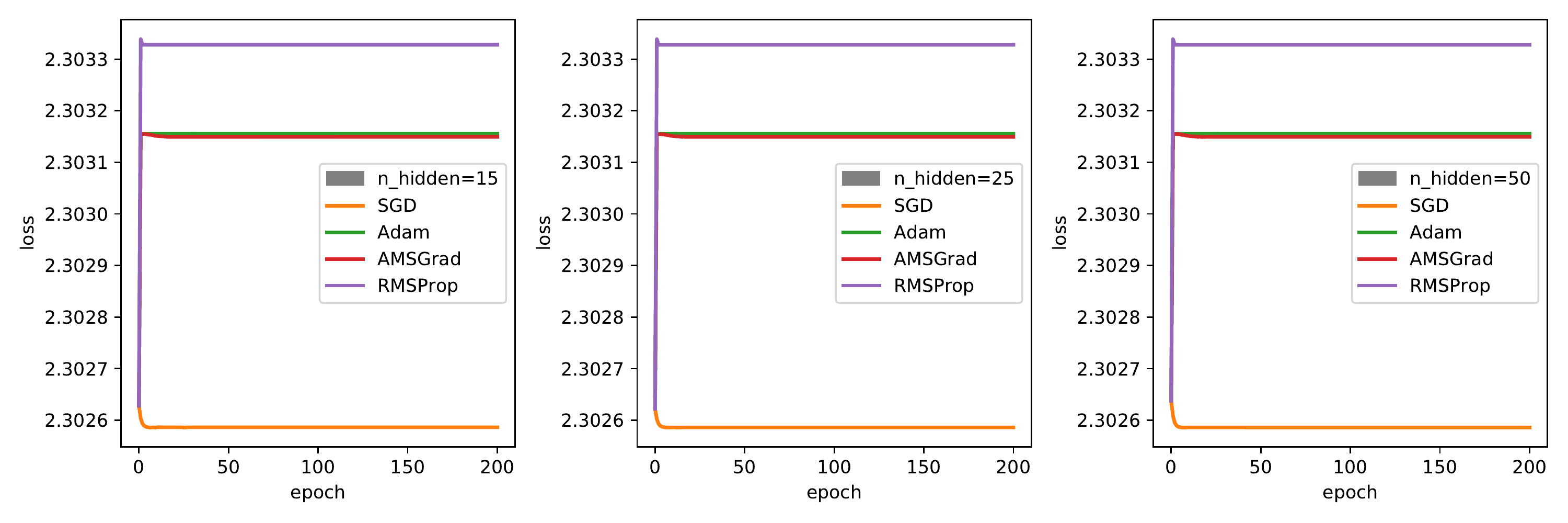}
    \caption{Training simple MLPs with $\mathcal{N}(-1, 0.01)$ initialization in Example 4. Top three: MNIST; Bottom three: CIFAR-10.}
    \label{fig:example4_apdx}
\end{figure}

\end{document}